\RequirePackage{fix-cm}
\documentclass[smallextended]{svjour3}       % onecolumn (second format)
\smartqed  % flush right qed marks, e.g. at end of proof
\usepackage{graphicx}

\def\x{\mathbf{x}}

\usepackage{xcolor}
\pagecolor[rgb]{0.9, 0.99, 0.9}

\usepackage{amsfonts,amssymb,amsmath} %%%add
\usepackage{epstopdf,mathrsfs} %%%add
\begin{document}

\title{Local discontinuous Galerkin method for the Backward Feynman-Kac Equation\thanks{ This work was supported by the National Natural Science Foundation of China under Grant No. 12071195, and the AI and
	Big Data Funds under Grant No. 2019620005000775.}
}


\author{Dong Liu         \and
        Weihua Deng %etc.
}


\institute{Dong Liu \at
               School of Mathematics and Statistics, Gansu Key Laboratory of Applied Mathematics and Complex Systems, Lanzhou University, Lanzhou 730000, P.R. China. \\
%              Tel.: +123-45-678910\\
%              Fax: +123-45-678910\\
              \email{dliu19@lzu.edu.cn}           %  \\
           \and
           Weihua Deng \at
              School of Mathematics and Statistics, Gansu Key Laboratory of Applied Mathematics and Complex Systems, Lanzhou University, Lanzhou 730000, P.R. China.\\
 \email{dengwh@lzu.edu.cn} }

\date{Received: date / Accepted: date}

\maketitle

\begin{abstract}
Anomalous diffusions are ubiquitous in nature, whose functional distributions are governed by the backward Feynman-Kac equation. In this paper, the local discontinuous Galerkin (LDG) method is used to solve the 2D backward Feynman-Kac equation in a rectangular domain. The spatial semi-discrete LDG scheme of the equivalent form (obtained by Laplace transform) of the original equation is established. After discussing the properties of the fractional substantial calculus, the stability and optimal convergence rates $O(h^{k+1})$ of the semi-discrete scheme are proved by choosing an appropriate generalized numerical flux. The $L1$ scheme on the graded meshes is used to deal with the weak singularity of the solution near the initial time. Based on the theoretical results of a semi-discrete scheme, we investigate the stability and convergence of the fully discrete scheme, which shows the optimal  convergence rates $O(h^{k+1}+\tau^{\min\{2-\alpha,\gamma\delta\}})$. Numerical experiments are carried out to show the efficiency and accuracy of the proposed scheme. In addition, we also verify the effect of the central numerical flux on the convergence rates and the condition number of the coefficient matrix.
\end{abstract}
\section{Introduction}\label{Sec1}
The origin of Feynman-Kac  transform can be traced back to Richard Feynman's research on ``path integrals" in the 1940s. Later, Mark Kac realized that the solution of Schr\"{o}ndinger equation (heat equation with external potential term) describing the functional distribution of diffusion motion can be obtained by this transformation \cite{Kac-1949-65}. The functional of anomalous diffusion has attracted extensive attention of physicists with the in-depth study of non-Brownian motion and anomalous diffusion. Similar to the functional of Brownian motion, the functional of anomalous diffusion can be defined as
\begin{equation}\label{founctional}
	A_t=\int_0^t\kappa(Y_{E_s})ds,
\end{equation}
where $ \kappa(\x) $ is a bounded function on the state space of the stochastic process $Y_{E_t}$ with $E_t$ being the inverse of the driftless subordinator with L\'evy measure $\mu(dx):=-d \omega(x)$ that is independent of the Markov process $Y_t$.

The distribution of the functional defined by \eqref{founctional} is governed by the 2D backward equation for Feynman-Kac transform with the abstract form \cite{Chen-2021-53}
\begin{equation}\label{first equation}%所有公式加标点符号   %ctrl+T\ctrl+U 添加\取消 注释
	\partial_{t}^{\omega,\kappa(\x)}u(t,\x)=\mathcal{L} u(t,\x)-\kappa(\x)I_{t}^{\omega,\kappa(\x)}u(t,\x) ~ \textrm{on}~ (0,T]\times\Omega, 	
\end{equation}
where $\omega(t)$ is an unbounded right continuous decreasing function on $(0,\infty)$ that is
integrable on $(0,1]$ and has $\lim_{t\rightarrow+\infty}\omega(t)=0$, $ \mathcal{L} $ is the generator of the stochastic process $Y_t$, $ \partial_{t}^{\omega,\kappa(x)}u(t,\x) $ denotes the generalized time fractional derivative, defined by
\begin{equation}\label{TFD}
	\partial_{t}^{\omega,\kappa(\x)}u(t,\x)=\frac{\partial}{\partial t}I_{t}^{\omega,\kappa(\x)}(u(t,\x)-u(0,\x))
\end{equation}
with $I_{t}^{\omega,\kappa(\x)}$ being the generalized time fractional integral
\begin{equation}\label{TFI}
	I_{t}^{\omega,\kappa(\x)}u(t,\x):=\int_{0}^{t}e^{-\kappa(\x)(t-\tau)}\omega(t-\tau)u(\tau,\x)d\tau.
\end{equation}
The boundary condition of (\ref{first equation}) is periodic and the initial condition is
\begin{equation}\label{first equation2}
	u(0,\x)=u_{0}(\x) , \quad x\in\Omega,	
\end{equation}
where the rectangular region $ \Omega \subset \mathbb{R}^2$.

The main challenges for analyzing and solving \eqref{first equation} come from the spatiotemporal coupling and nonlocality of the generalized time fractional derivative \eqref{TFD}; it is harder to get analytical solution. So, finding an effective numerical method to solve \eqref{first equation} seems to be urgent.
%The spatiotemporal coupling and nonlocality of the generalized time fractional derivative \eqref{TFD} in the Eq.~\eqref{first equation} result in some difficulties to our analysis. Therefore, it is urgent to find an effective numerical method to solve the Eq.~\eqref{first equation}.
Currently, there are many discussions for the numerical algorithms of fractional partial differential equations, such as discontinuous Galerkin method (DG) \cite{Deng-2013-47,Nie-2021-104,Sun-2020-365}, finite difference method \cite{Deng-2015-62,Sun-2021-161}, finite element method \cite{Deng-2008-47}, and spectral method \cite{Tian-2014-30}, etc. To the best of our knowledge, the DG method for \eqref{first equation} has not been
discussed, and the main challenge lies in that $ \kappa(\x) $ may be negative. In the paper, this problem has been overcome by using some special techniques given in Lemma \ref{L2-1} and Lemma \ref{L2-2} (see below).

In 1973, the DG method was used to solve the neutron transport equation by Reed and Hill and it allows the basis functions in each element to be relatively independent and requires information exchange by defining the numerical fluxes at the boundaries of adjacent elements. Since the 1990s, the Runge-Kutta discontinuous Galerkin method proposed by Cockburn and Shu \cite{Xu-2014-278,Zhang-2004-42} has been widely used. The DG method can flexibly deal with many problems that are not easy to deal with by continuous finite element method and it's conducive to the formation of adaptive mesh \cite{Wang-2020-38} and parallel computing. Nowadays, the DG method is widely used to solve the equations in physics, chemistry, biology, and atmosphere science, generally having the property of strong convection. Initially, the DG method was mainly used to solve and analyze first-order problems.
 %and it's necessary to transform the higher derivative into the first derivative of a nonlinear term  in higher order problems. Obviously, the DG method is more complex to deal with high-order problems.
Due to the convenience of solving the first-order problem by DG method, Bassi and Rebay transformed the high-order equation into a first-order system by introducing auxiliary variables to solve the Navier-Stokes equation \cite{Bassi-1997-131}, and finally they solved the high-order problem conveniently by using the DG method. This method is called ``local discontinuous Galerkin method" (LDG)  with the reason that auxiliary variables can be solved in local elements. Later, Cockburn and Shu  \cite{Cockburn-1998-35} established the theoretical framework of LDG method.

Initially, the LDG method was used by Cockburn and Shu \cite{Cockburn-2007-32,Cockburn-1998-35} to solve the convection diffusion problem. Deng and Hesthaven \cite{Deng-2013-47} established a theoretical framework for LDG method to solve the spatial fractional diffusion equation in 2013. Since then, the LDG method has been widely used to solve the fractional diffusion equation with integral fractional Laplacian \cite{Nie-2021-104}, the 2D fractional diffusion equation \cite{Qiu-2015-298}, the fractional telegraph equation \cite{Wei-2014-51}, the time tempered fractional diffusion equation \cite{Sun-2020-365}, the fractional convection diffusion equation \cite{Xu-2014-52}, the fractional Burgers equation \cite{Mao-2017-336}, and the fractional Allen-Cahn equation \cite{Xia-2008-5}, etc. This method inherits the flexibility of DG method and it can better retain the physical properties of the model when dealing with models with poor regularity. The combination of the LDG method with the adaptive strategy can better reflect its advantages.

Although the DG method can be used to discretize the time derivatives \cite{Mustapha-2016-73}, more often it is used to approximate spatial operators. The existing discussions for the semi-discrete LDG scheme of spatial fractional partial differential equations are usually with the time classical derivatives  \cite{Deng-2013-47,Nie-2021-104,Xia-2008-5,Xu-2014-52}.
%Based on the existing analysis results of the time classical derivative, the semi-discrete LDG scheme for spatial fractional partial differential equations with the time classical derivatives is  usually only discussed \cite{Deng-2013-47,Nie-2021-104,Xia-2008-5,Xu-2014-52}.
Because of the nonlocal property of fractional derivative, the theoretical analysis of LDG method for time fractional partial differential equation (TFPDE) is often directly based on its fully discrete scheme \cite{Huang-2020-151,Qiu-2015-298,Wei-2014-51,Wei-2014-38}. As far as we know, there are few studies on LDG semi-discrete schemes for TFPDE. Although the properties of continuous Riemann-Liouville time fractional derivatives provide ideas for studying the properties of other types of time fractional derivatives, there are still great challenges when $ \kappa(\x) $ is negative.  To overcome this problem, we will use Fourier transform and Cauchy integral theorem.

If the solution of the equation to be solved is sufficiently regular, the convergence order of $L1$ scheme of Caputo derivative on uniform meshes can reach $( 2-\alpha) $ in theory \cite{Liu-2007-225}. Since the solution of Caputo fractional derivative problem has weak singularity near the initial value \cite{McLean-2010-53,Stynes-2016-19}, i.e., $|u^{(l)}(t)|\leq C_u(1+t^{\alpha-l})$, \,$l=0,1,2$, the $L1$ scheme cannot reach the convergence order of $( 2-\alpha) $ on uniform meshes. Fortunately, the non-uniform mesh method can better overcome this problem. Stynes et al. \cite{Stynes-2017-55} proposed the $L1$ scheme based on graded meshes in order to overcome the weak singularity of the solution at the initial time in the time fractional reaction-diffusion equation. Huang and Stynes \cite{Huang-2020-367} proposed a finite element scheme based on graded meshes for the time fractional initial boundary value problem. Li et al. \cite{Li-2016-316} established the finite difference scheme on non-uniform meshes for nonlinear fractional differential equations. In this paper, we also use graded meshes to overcome the weak singularity of the solution of \eqref{first equation} near the initial value.

In the following, we take
\begin{equation*}
	\mathcal{L} =\Delta \quad \textrm{ and } \quad \omega(\tau)=\frac{\tau^{-\alpha}}{\Gamma(1-\alpha)},\quad 0<\alpha<1,
\end{equation*}
where $ \Delta $ is the Laplacian operator, i.e., the generator of Brownian motion, and $\Gamma(\cdot)$ is the Gamma function. Then Eq.~\eqref{first equation} is called the backward Feynman-Kac equation \cite{Carmi-2011-84,Carmi-2010-141,Deng-2015-62,urgeman-2009-103}.
%	, and ${}_0^C D_t^{\alpha,\kappa(\x)}u(t,\x)$ is called the time Caputo fractional substantial derivative of $u(t,\x)$ \cite{Deng-2015-62}. When $\kappa(\x)$ is a positive contant, ${}_0^C D_t^{\alpha,\kappa(\x)}$ is so-called the tempered Caputo fractional derivative \cite{Sun-2020-365,Wang-2020-38}.
In fact, according to the definition of the generalized time fractional derivative in \eqref{TFD}, $\partial_{t}^{\omega,\kappa(\x)}u(t,\x)$ can be rewritten as
\begin{equation}\label{TFD1}
	\partial_{t}^{\omega,\kappa(\x)}u(t,\x)=\frac{1}{\Gamma(1-\alpha)}\frac{\partial }{\partial  t}\int_{0}^{t}e^{-\kappa(\x)(t-\tau)}(t-\tau)^{-\alpha}(u(\tau,\x)-u(0,\x))d\tau.
\end{equation}
Then we can further get the equivalent form of  Eq.~\eqref{first equation} (see the Appendix)
\begin{equation}\label{equivalent}
	\Delta u(t,\x)=e^{-\kappa(\x)t}{}_0^C D_t^{\alpha}(e^{\kappa(\x)t}u(t,\x)):={}_0^C D_t^{\alpha,\kappa(\x)}u(t,\x),
\end{equation}
where $ {}_0^C D_t^{\alpha}u(t,\x) $ is the Caputo fractional derivative of order $ \alpha \in(0,1) $, defined by
\begin{equation*}
	{}_0^C D_t^{\alpha}u(t,\x)=\frac{1}{\Gamma(1-\alpha)}\int_{0}^{t}(t-\tau)^{-\alpha}\partial_{\tau}u(\tau ,\x)d\tau.
\end{equation*}	
This paper will focus on the construction and analysis of the numerical scheme for Eq.~\eqref{equivalent} with
periodic boundary and initial conditions \eqref{first equation2}.

Since $\kappa(\x)$ is a bounded function on $\Omega$, naturally the following hold: %we can give two conditions about $\kappa(\x)$.

 $({\textbf{\textit{\romannumeral1}}})$ There exists a positive constant $C_{\kappa}$ such that
	\begin{equation}\label{assumptions1}
		|\kappa(\x)|\leq C_{\kappa}, \quad \x\in\Omega;
	\end{equation}

$({\textbf{\textit{\romannumeral2}}})$ There exist two positive constants $C_{min}$ and $C_{max}$ such that
	\begin{equation}\label{assumptions2}
		C_{min}\leq e^{-\kappa(\x)t}\leq C_{max}, \quad (t,\x)\in[0,T]\times\Omega.
	\end{equation}

The rest of this paper is organized as follows. The properties of fractional substantial calculus are proved in Section \ref{Sec2}. In Section \ref{Sec3}, we present the spatial semi-discrete LDG scheme of the equivalent form of the original equation using the generalized alternating numerical flux for two-dimensional space; and the $L^{2}$-stability and a priori error estimate are also proposed in Theorem \ref{Th3_1} and Theorem \ref{Th3_2}, respectively. In Section \ref{Sec4}, the fully discrete scheme is established by the L1 scheme of fractional substantial derivative on graded mesh. Based on the theoretical results of semi-discrete scheme, the fully discrete scheme is theoretically analyzed. Section \ref{Sec5} contains some numerical results. The paper concludes with some discussions in the last section. In the Appendix, we provide the proof for the equivalent form of the original equation.

\section{Notations and some preliminaries}
\label{Sec2}

%\mathfrak{}
In this section, we introduce the fractional substantial calculuses, analyze their properties, and also introduce the used function spaces.

\subsection{Properties of the time fractional substantial calculus}
\label{Sec2_1}
%as required. Don't forget to give each section
%and subsection a unique label (see Sect.~\ref{sec:2}).
%\paragraph{Paragraph headings} Use paragraph headings as needed.
%\begin{equation}
%	a^2+b^2=c^2
%\end{equation}

 First, we introduce the time fractional substantial calculus \cite{Chen-2015-37,Deng-2015-62,Wang-2020-38}.
\begin{definition}
	For any $\alpha>0$, the time fractional substantial integral of the function $ u(t) $ defined on $[0,\infty)$ is given by
	\begin{equation*}\label{sub_int}
		{}_0I_{t}^{\alpha,\kappa(\x)}u(t)=e^{-\kappa(\x)t}{}_0I_t^{\alpha}[e^{\kappa(\x)t}u(t)],
	\end{equation*}
	where $\kappa(\x)$ is a prescribed function in \eqref{founctional}.
	Here $ {}_0I_t^{\alpha}u(t) $ is the Riemann-Liouville fractional integral of order $ \alpha $, which is defined by
	\begin{equation*}
		{}_0I_t^{\alpha}u(t)=\frac{1}{\Gamma(\alpha)}\int_{0}^{t}(t-\tau)^{\alpha-1}u(\tau)d\tau.
	\end{equation*}	
	
\end{definition}

\begin{definition}
	For any $\alpha\in(0,1)$, the time fractional substantial derivative of the function $ u(t) $ defined on $[0,\infty)$ is given by
	\begin{equation*}\label{sub_der}
		\begin{split}
			{}_0^R D_t^{\alpha,\kappa(\x)}u(t)&=e^{-\kappa(\x)t}{}_0^R D_t^{\alpha}[e^{\kappa(\x)t}u(t)],\\
		\end{split}
	\end{equation*}
	where $\kappa(\x)$ is a prescribed function in \eqref{founctional}.
	Here  $ {}_0^RD_t^{\alpha}u(t) $ is the Riemann-Liouville fractional derivative of order $ \alpha\in(0,1)$, which is defined by
	\begin{equation*}
		{}_0^RD_t^{\alpha}u(t)=\frac{1}{\Gamma(1-\alpha)}\dfrac{d}{dt}\int_{0}^{t}(t-\tau)^{-\alpha}u(\tau)d\tau.
	\end{equation*}	
\end{definition}

\begin{lemma}\label{F_T}	
	For $u(t)\in L^2(\mathbb{R})$ and $\alpha>0$, $\kappa(\x)$ defined on $\Omega$, it holds that
	\begin{equation*}\label{F_T_I}
		\begin{aligned}
			&\mathscr{F}[e^{\kappa(\x)t}u(t)](\omega)=\tilde{u}(\kappa(\x)+i\omega),\\
			&\mathscr{F}[{}_{-\infty} I_{t}^{\alpha,\kappa(\x)}u(t)](\omega)=(\kappa(\x)+i\omega)^{-\alpha}\tilde{u}(\omega).
		\end{aligned}
	\end{equation*}
	
	If $u(t)\in H^{\alpha}(\mathbb{R}) $ further, then	
	\begin{equation*}\label{F_T_D}
		\mathscr{F}[{}_{-\infty}^{\quad R} D_t^{\alpha,\kappa(\x)}u(t)](\omega)=(\kappa(\x)+i\omega)^{\alpha}\tilde{u}(\omega),
	\end{equation*}	
	where $i=\sqrt{-1}$ and $H^{\alpha}(\mathbb{R})$ is the fractional Sobolev space with the norm
	\begin{equation*}
		\|u(t)\|_{H^{\alpha}(\mathbb{R})}^2:=\int_{-\infty}^{\infty}(1+|\omega|^{2\alpha})|\tilde{u}|^2d\omega.
	\end{equation*}
	Here $\mathscr{F}$ denotes Fourier transform operator, i.e.,
	\begin{equation*}
		\tilde{u}(\omega):=	\mathscr{F}[u(t)](\omega)=\dfrac{1}{\sqrt{2\pi}}\int_{-\infty}^{\infty}e^{-i \omega t}u(t)dt.
	\end{equation*}

\end{lemma}

To keep the technical details in the forthcoming numerical analyses at a moderate level, similar to \cite{McLean-2012-34,Mustapha-2016-73,Mustapha-2014-34}, the properties of the time fractional substantial calculus can be first proved. To be convenient, $u(t)$ and $v(t)$  are used instead of $ u(t,\x) $ and $ v(t,\x) $ in this section, respectively. Next, we introduce two Lemmas.
\begin{lemma}\label{L2-1}
	Let $u(t)$ be a piecewise $C^{1}$ function on $[0,T]$ and $ \kappa(\x) $ be a given function on $\Omega$. Then
	\begin{equation*}\label{semi-discrete Stability3-1}
		\int_{0}^{T}u(t)\cdot {}_0^R D_t^{\alpha,\kappa(\x)}u(t)dt\geq C_{\alpha}T^{-\alpha}\int_{0}^{T}u^{2}(t)dt, \quad 0<\alpha<1,
	\end{equation*}
	where $C_{\alpha}=\dfrac{1}{\alpha+1}\Big(\dfrac{\alpha\pi}{\alpha+1}\Big)^{\alpha}\cdot\min{\{C_{max}^{-(4+4\alpha)},1\}} $. Here $ C_{max} $ is defined in \eqref{assumptions2}.
\end{lemma}
\begin{proof}
	First, we assume that there exists a set $ \Omega_{1}\subset \Omega $  such that
	\begin{equation*}\label{semi-discrete Stability3-1-2}
		\left\{
		\begin{aligned}
			& \kappa(\x)\geq 0 &&\textrm{for}~~ \x \in \Omega_{1}, \\
			& \kappa(\x)< 0  &&\textrm{for}~~ \x \in \Omega\backslash\Omega_{1}.\\
		\end{aligned} \right.
	\end{equation*}
	
	Let $ v(t) $ be a piecewise $ C^{1} $ function on $ [0,1] $. Thus,   ${}_0^RD_t^{\alpha}(e^{\kappa(\x)Tt}v(t)) $ is continuous except
	for weak singularities at the breakpoints of $ e^{\kappa(\x)Tt}v(t) $. We extend $ v(t) $ by zero outside the interval $ [0,1] $. Then, one can get %the following relation
	\begin{equation*}
		\tilde{v}(y)=\dfrac{1}{\sqrt{2\pi}}\int_{-\infty}^{\infty}e^{-iyt} v(t)dt=\dfrac{1}{\sqrt{2\pi}}\int_{0}^{\infty}e^{-iyt} v(t)dt=\dfrac{1}{\sqrt{2\pi}}\hat{v}(iy),	
	\end{equation*}	
	where $\hat{v}$ is the Laplace transform of $ v $ and $i=\sqrt{-1}$.
	
	By Plancherel's theorem, the fact that $ \overline{\hat{v}(iy) }=\hat{v}(-iy) $ ($ v $ is a real-valued function), we find that	
	\begin{equation}\label{semi-discrete Stability3-1-1}
		\begin{split}
			\int_{0}^{1}v(t)\cdot {}_0^R D_t^{\alpha,\kappa(\x)T}v(t)dt&=\int_{-\infty}^{\infty} v(t) \cdot  {}_0^R D_t^{\alpha,\kappa(\x)T}v(t)dt\\
%			&=\dfrac{1}{2\pi}\int_{-\infty}^{\infty}\overline{\hat{v}(iy)}(iy+\kappa(\x)T)^{\alpha}\hat{v}(iy)dy\\
			&=\dfrac{1}{2\pi}\int_{-\infty}^{\infty}(iy+\kappa(\x)T)^{\alpha}|\hat{v}(iy)|^{2}dy,\\
		\end{split}		
	\end{equation}
	where Lemma\,\ref{F_T} is used. Here $ \overline{\hat{v}(iy) } $ represents the complex conjugate of $ \hat{v}(iy) $.	
	
	By convention, denote the left hand side of \eqref{semi-discrete Stability3-1-1} by $\uppercase\expandafter{\romannumeral1}_1 $. In what follows, we only need to consider two cases for $\kappa(\x)$:
	
	{\bf Case I}: $ \kappa(\x)\geq 0 \
	\textrm{for}~ \x \in \Omega_{1} $. By \eqref{semi-discrete Stability3-1-1}, one can obtain
	\begin{equation}\label{semi-discrete Stability3-1-3}
		\begin{split}
			\uppercase\expandafter{\romannumeral1}_1&=\dfrac{1}{\pi}\int_{0}^{\infty}Re[(iy+\kappa(\x)T )^{\alpha}]|\hat{v}(iy)|^{2}dy\\
			&=\dfrac{1}{\pi}\int_{0}^{\infty}|iy+\kappa(\x)T |^{\alpha}\cos\theta|\hat{v}(iy)|^{2}dy\\
			&\geq\dfrac{1}{\pi}\cos\Big(\dfrac{\pi\alpha}{2}\Big)\int_{0}^{\infty}y^{\alpha}\cdot|\hat{v}(iy)|^{2}dy,\\
		\end{split}		
	\end{equation}
	where $ \theta\leq\dfrac{\pi\alpha}{2} $ is the argument of a complex number $ (iy+\kappa(\x)T )^{\alpha} $ and $ Re[z] $ is the real part of a complex number $ z $.
	
	For any $ \varepsilon>0 $, using the Cauchy-Schwarz inequality, one can get
	\begin{equation}\label{semi-discrete Stability3-2}	
		\int_{0}^{\varepsilon}|\hat{v}(iy)|^{2}dy=\int_{0}^{\varepsilon}\Big|\int_{0}^{1}e^{-iyt}v(t)dt\Big|^{2}dy\leq\varepsilon\int_{0}^{1}v^{2}(t)dt.
	\end{equation}
	By Plancherel's theorem for $ v(t) $, there is
	\begin{equation*}\label{semi-discrete Stability3-3}
		\begin{split}
			\int_{0}^{1}v^{2}(t)dt
%			&=\dfrac{1}{\pi}\int_{0}^{\infty}|\hat{v}(iy)|^{2}dy\\	
			&\leq\dfrac{\varepsilon }{\pi}\int_{0}^{1}v^{2}(t)dt+\dfrac{1}{\pi}\int_{\varepsilon}^{\infty}|\hat{v}(iy)|^{2}dy,\\
		\end{split}
	\end{equation*}
	where the property \eqref{semi-discrete Stability3-2} is used in the last step.
	
	For $ 0< \varepsilon < \pi$, it follows that
	\begin{equation}\label{semi-discrete Stability3-5}
		\begin{split}	
			\Big(1-\dfrac{\varepsilon}{\pi}\Big)\int_{0}^{1}v^{2}(t)dt&=\dfrac{1}{\pi}\int_{\varepsilon}^{\infty}|\hat{v}(iy)|^{2}dy\\
			&\leq\dfrac{1}{\pi}\int_{\varepsilon}^{\infty}\Big(\dfrac{y}{\varepsilon}\Big)^{\alpha}|\hat{v}(iy)|^{2}dy\\
			&\leq\dfrac{1}{\pi\varepsilon^{\alpha}}\int_{0}^{\infty}y^{\alpha}|\hat{v}(iy)|^{2}dy.\\
		\end{split}
	\end{equation}
	
Combining \eqref{semi-discrete Stability3-1-3} and \eqref{semi-discrete Stability3-5} leads to
	\begin{equation*}\label{semi-discrete Stability3-6}
		\begin{split}
			\int_{0}^{1}v(t)\cdot {}_0^R D_t^{\alpha,\kappa(\x)T}v(t)dt&\geq \Big(\varepsilon^{\alpha}-\dfrac{\varepsilon^{\alpha+1}}{\pi}\Big)\cos\Big(\dfrac{\pi\alpha}{2}\Big)\int_{0}^{1}v^{2}(t)dt.\\
		\end{split}
	\end{equation*}
%	where $\varphi_{1}(\varepsilon)=\dfrac{1}{\pi}(\pi\varepsilon^{\alpha}-\varepsilon^{\alpha+1}) $.
%	
%	Since $\varphi_{1}(\varepsilon)_{\max}
%%	=\varphi_{1}\Big(\dfrac{\alpha\pi}{\alpha+1}\Big)
%	=\dfrac{1}{\alpha+1}\Big(\dfrac{\alpha\pi}{\alpha+1}\Big)^{\alpha} $,
	
	Hence
	\begin{equation*}\label{semi-discrete Stability3-7}
		\begin{split}
			\int_{0}^{1}v(\tau)\cdot {}_0^R D_\tau^{\alpha,\kappa(\x)T}v(\tau)d\tau&\geq C_{\alpha1}\int_{0}^{1}v^{2}(\tau)d\tau,\\
		\end{split}
	\end{equation*}
	where $ C_{\alpha1}=\dfrac{1}{\alpha+1}\Big(\dfrac{\alpha\pi}{\alpha+1}\Big)^{\alpha}\cos\Big(\dfrac{\pi\alpha}{2}\Big) $.
	
	{\bf Case II}:  $ \kappa(\x)< 0 \
	\textrm{for}~ \x \in \Omega\backslash\Omega_{1} $. If $ y>\kappa(\x)T\cdot\tan\Big(\dfrac{\pi}{2\alpha}\Big) $, there exists
	\begin{equation*}\label{semi-discrete Stability3-1-4}
		Re[(iy+\kappa(\x)T )^{\alpha}]>0.
	\end{equation*}
	
	Hence, by \eqref{semi-discrete Stability3-1-3}, it holds that
	\begin{equation}\label{semi-discrete Stability3-1-5}
		\lim_{y\rightarrow+\infty}Re[(iy+\kappa(\x)T )^{\alpha}]|\hat{v}(iy)|^{2}=0.
	\end{equation}
	
	In complex plane, we define
	\begin{equation}\label{semi-discrete Stability3-1-6}
		\begin{aligned}
			&\Gamma_{1}:=\{iy+\kappa(\x)T \big| -\infty<y<+\infty\},\\
			&\Gamma_{2}:=\{iy-\kappa(\x)T  \big| -\infty<y<+\infty\},\\
			&\Gamma_{3}:=\{iy_{0}+x  \big|  \ \kappa(\x)T\leq x\leq-\kappa(\x)T ,\quad y_{0}>0\}, \\
			&\Gamma_{4}:=\{-iy_{0}+x  \big| \ \kappa(\x)T\leq x\leq-\kappa(\x)T ,\quad y_{0}>0\}, \\
			&\Sigma:=\{z  \big|\ |Re(z)|\leq-\kappa(\x)T\},\\
		\end{aligned}
	\end{equation}
	where $ \Gamma_{i} (i=1,2,3,4) $ are with the directions being the same as the ones of the corresponding coordinate axes.
	
	By using \eqref{semi-discrete Stability3-1-1}, it has
	\begin{equation*}\label{semi-discrete Stability3-1-7}
		\uppercase\expandafter{\romannumeral1}_1=\dfrac{1}{2\pi}\int_{\Gamma_{1}}z^{\alpha}\cdot|\hat{v}(z-\kappa(\x)T)|^{2}dz.
	\end{equation*}
	Since $ v(t) $ has a compact  support in $ \mathbb{R} $, i.e., $ v(t) $ is zero outside of $[0,1]$, we know  that $ z^{\alpha}\cdot|\hat{v}(z-\kappa(\x)T)|^{2} $ is analytic in the strip
	shape area $ \Sigma $. By Cauchy integral theorem, there exists
	\begin{equation*}\label{semi-discrete Stability3-1-9}
		\int_{\Gamma_{1}-\Gamma_{2}}z^{\alpha}\cdot|\hat{v}(z-\kappa(\x)T)|^{2}dz+\lim_{y_{0}\rightarrow+\infty}\int_{\Gamma_{3}-\Gamma_{4}}z^{\alpha}\cdot|\hat{v}(z-\kappa(\x)T)|^{2}dz=0,
	\end{equation*}
	where $ -\Gamma_{2} $ means the direction is opposite to $ \Gamma_{2} $. By \eqref{semi-discrete Stability3-1-5} and \eqref{semi-discrete Stability3-1-6}, we have
	\begin{equation*}\label{semi-discrete Stability3-1-8}
		\begin{split}
			\lim_{y_{0}\rightarrow+\infty}\int_{\Gamma_{3}-\Gamma_{4}}z^{\alpha}\cdot|\hat{v}(z-\kappa(\x)T)|^{2}dz
%			&=	\lim_{y_{0}\rightarrow+\infty}2\int_{\Gamma_{3}}Re[z^{\alpha}]\cdot|\hat{v}(z-\kappa(\x)T)|^{2}dz\\
			&=2\int_{\Gamma_{3}}\lim_{y_{0}\rightarrow+\infty}Re[z^{\alpha}]\cdot|\hat{v}(z-\kappa(\x)T)|^{2}dz\\
			&=0.\\	
		\end{split}	
	\end{equation*}
	Thus,
	\begin{equation}\label{semi-discrete Stability3-1-10}
		\begin{split}
			\uppercase\expandafter{\romannumeral1}_1&=\dfrac{1}{2\pi}\int_{\Gamma_{2}}z^{\alpha}\cdot|\hat{v}(z-\kappa(\x)T)|^{2}dz\\
			&=\dfrac{1}{2\pi}\int_{-\infty}^{\infty}(iy-\kappa(\x)T)^{\alpha}\cdot|\hat{v}(iy-2\kappa(\x)T)|^{2}dy.\\
		\end{split}
	\end{equation}
	Similar to \eqref{semi-discrete Stability3-1-3}, there exists
	\begin{equation}\label{semi-discrete Stability3-1-11}
		\begin{split}
			\uppercase\expandafter{\romannumeral1}_1&\geq\dfrac{1}{\pi}\cos\Big(\dfrac{\pi\alpha}{2}\Big)\int_{0}^{\infty}y^{\alpha}\cdot|\hat{v}(iy-2\kappa(\x)T)|^{2}dy.\\
		\end{split}
	\end{equation}
	For any $ \varepsilon>0 $, using the Cauchy-Schwarz inequality leads to
	\begin{equation}\label{semi-discrete Stability3-1-12}
		\begin{split}	
			\int_{0}^{\varepsilon}|\hat{v}(iy-2\kappa(\x)T)|^{2}dy&=	\int_{0}^{\varepsilon}\Big|\int_{0}^{1}e^{-(iy-2\kappa(\x)T)t}v(t)dt\Big|^{2}dy\\
			&\leq \varepsilon\int_{0}^{1}v^{2}(t)dt,\\
		\end{split}
	\end{equation}	
	where \eqref{assumptions2} is used. Next, invoking Plancherel's theorem for $ e^{2\kappa(x)Tt}v(t) $ followed by an application of \eqref{semi-discrete Stability3-1-12},
	it holds that
	\begin{equation*}\label{semi-discrete Stability3-1-13}
		\begin{split}
			\dfrac{1}{C_{max}^{4}}\int_{0}^{1}v^{2}(t)dt&\leq\int_{0}^{1}e^{4\kappa(\x)Tt}v^{2}(t)dt\\
			&=\dfrac{1}{\pi}\int_{0}^{\infty}|\hat{v}(iy-2\kappa(\x)T)|^{2}dy\\	
			&\leq\dfrac{\varepsilon }{\pi }\int_{0}^{1}v^{2}(t)dt+\dfrac{1}{\pi}\int_{\varepsilon}^{\infty}|\hat{v}(iy-2\kappa(\x)T)|^{2}dy.\\
		\end{split}
	\end{equation*}
For $ 0< \varepsilon < \dfrac{\pi}{ C_{max}^{4}}$, it follows that
	\begin{equation}\label{semi-discrete Stability3-1-14}
		\begin{split}	
			\Big(\dfrac{1}{C_{max}^{4}}-\dfrac{\varepsilon }{\pi }\Big)\int_{0}^{1}v^{2}(t)dt&=\dfrac{1}{\pi}\int_{\varepsilon}^{\infty}|\hat{v}(iy-2\kappa(\x)T)|^{2}dy\\
%			&\leq\dfrac{1}{\pi}\int_{\varepsilon}^{\infty}\Big(\dfrac{y}{\varepsilon}\Big)^{\alpha}|\hat{v}(iy-2\kappa(\x)T)|^{2}dy\\
			&\leq\dfrac{1}{\pi\varepsilon^{\alpha}}\int_{0}^{\infty}y^{\alpha}|\hat{v}(iy-2\kappa(\x)T)|^{2}dy.\\
		\end{split}
	\end{equation}
By \eqref{semi-discrete Stability3-1-11} and \eqref{semi-discrete Stability3-1-14}, we obtain
	\begin{equation*}\label{semi-discrete Stability3-1-15}
		\begin{split}
			\int_{0}^{1}v(t)\cdot {}_0^R D_t^{\alpha,\kappa(\x)T}v(t)dt&\geq \Big(\dfrac{\varepsilon^{\alpha}}{C_{max}^{4}}-\dfrac{\varepsilon^{\alpha+1}}{\pi}\Big)\cos\Big(\dfrac{\pi\alpha}{2}\Big)\int_{0}^{1}v^{2}(t)dt.\\
		\end{split}
	\end{equation*}
%	where $\varphi_{2}(\varepsilon)=\dfrac{1}{\pi}\Big(\dfrac{\pi\varepsilon^{\alpha}}{C_{max}^{4}}-\varepsilon^{\alpha+1}\Big) $.
%	
%	Since $\varphi_{2}(\varepsilon)_{\max}=\dfrac{1}{(\alpha+1)C_{max}^{4}}\Big(\dfrac{\alpha\pi }{(\alpha+1)C_{max}^{4}}\Big)^{\alpha} $,
Then one can easily see that	
	\begin{equation*}\label{semi-discrete Stability3-1-16}
		\begin{split}
			\int_{0}^{1}v(\tau)\cdot {}_0^R D_\tau^{\alpha,\kappa(\x)T}v(\tau)d\tau&\geq C_{\alpha2}\int_{0}^{1}v^{2}(\tau)d\tau,\\
		\end{split}
	\end{equation*}
	where $ C_{\alpha2}=\dfrac{1}{(\alpha+1)C_{max}^{4}}\Big(\dfrac{\alpha\pi }{(\alpha+1)C_{max}^{4}}\Big)^{\alpha}\cos\Big(\dfrac{\pi\alpha}{2}\Big) $.
	
Combining the results of {\bf Case I} and  {\bf Case II} lead to
	\begin{equation}\label{semi-discrete Stability3-1-17}
		\int_{0}^{1}v(\tau)\cdot {}_0^R D_\tau^{\alpha,\kappa(\x)T}v(\tau)d\tau\geq C_{\alpha}\int_{0}^{1}v^{2}(\tau)d\tau.
	\end{equation}
Finally, by the scaling argument, we arrive at
	%		\begin{equation*}\label{semi-discrete Stability3-9}
	%			\begin{split}
	%				\int_{0}^{1}v(\tau)\cdot {}_0^R D_\tau^{\alpha,\kappa(\x)T}v(\tau)d\tau=T^{\alpha+1}\int_{0}^{T}u(t)\cdot {}_0^R D_t^{\alpha,\kappa(\x)}u(t)dt,
	%			\end{split}
	%		\end{equation*}
	%		and similary for the right-hand side of \eqref{semi-discrete Stability3-1-17},
	%		\begin{equation*}\label{semi-discrete Stability3-10}
	%			\int_{0}^{1}v^{2}(\tau)d\tau=	T\int_{0}^{T}u^{2}(t)dt,
	%		\end{equation*}
	%		where $ u(t)=v\Big(\dfrac{t}{T}\Big) $.
	
	%		Collecting the above equalities together, we arrive at
	\begin{equation*}\label{semi-discrete Stability3-11}
		\int_{0}^{T}u(t)\cdot {}_0^R D_t^{\alpha,\kappa(\x)}u(t)dt\geq C_{\alpha}T^{-\alpha}\int_{0}^{T}u^{2}(t)dt,
	\end{equation*}
	where $ u(t)=v(t/T) $.		
	Thus the proof is completed.
\end{proof}	
\begin{lemma} \label{L2-2}
	Let $u(t)$ and $v(t)$ be piecewise continuous on $[0,T]$ and $ \phi(\x) $ be a given function on $\Omega$. Then
	\begin{equation*}\label{semi-discrete Stability3-12}
		\begin{split}
			\Big|\int_{0}^{T}e^{-\phi(\x)t}u(t)\cdot v(t)dt\Big|^{2}&\leq \sec^{2}\Big(\dfrac{\pi\alpha}{2}\Big)\int_{0}^{T}v(t)\cdot {}_0^R D_t^{\alpha,\phi(\x)}v(t)dt\\
			&~~~\cdot\int_{0}^{T}e^{2(|\phi(\x)|-2\phi(\x))t}u(t)\cdot {}_0I_t^{\alpha,2(|\phi(\x)|-\phi(\x))}u(t)dt.
		\end{split}
	\end{equation*}
	
\end{lemma}
\begin{proof}
	First, we assume that there exists a set $ \Omega_{1}\subset \Omega $  such that
\begin{equation*}\label{semi-discrete Stability3-1-2}
	\left\{
	\begin{aligned}
		& \phi(\x)\geq 0 &&\textrm{for}~~ \x \in \Omega_{1}, \\
		& \phi(\x)< 0  &&\textrm{for}~~ \x \in \Omega\backslash\Omega_{1}.\\
	\end{aligned} \right.
\end{equation*}
	
	In what follows, we analyze two cases.
	
	{\bf Case I}: $ \phi(\x)\geq 0 \
	\textrm{for}~ \x \in \Omega_{1} $. The similar analysis can be found in \eqref{semi-discrete Stability3-1-1}. By	Plancherel's theorem, there exists
	\begin{equation}\label{semi-discrete Stability3-13}
		\begin{split}
			\int_{0}^{T}e^{-\phi(\x)t}u(t)\cdot v(t)dt&=\dfrac{1}{2\pi}\int_{-\infty}^{\infty}\overline{\hat{u}(z_1)}\cdot\hat{v}(iy)dy \\
			&\leq\dfrac{1}{2\pi}\int_{-\infty}^{\infty}|\overline{\hat{u}(z_1)}\cdot\hat{v}(iy)| dy, \\
		\end{split}
	\end{equation}
	where $z_1=iy+\phi(\x)$.
	
	For any $ \varepsilon>0 $, using the Cauchy-Schwarz inequality in the bounded interval implies that
	\begin{equation*}\label{semi-discrete Stability3-14}
		\begin{split}
			\varPhi^{2}(\varepsilon)
			&=\Big(\int_{-\varepsilon}^{\varepsilon}\Big|\overline{\hat{u}(z_1)}z_1^{-\alpha/2}\cdot z_1^{\alpha/2}\hat{v}(iy)\Big|dy\Big)^{2}  \\
			&\leq\int_{-\varepsilon}^{\varepsilon}\Big|z_1^{\alpha/2}\cdot\hat{v}(iy)\Big|^{2}dy\cdot\int_{-\varepsilon}^{\varepsilon}\Big|z_1^{-\alpha/2}\cdot\hat{u}(z_1)\Big|^{2}dy,\\  	
		\end{split}
	\end{equation*}
	where $\varPhi(\varepsilon)$ is defined by
	\begin{equation*}\label{semi-discrete Stability3-15}
		\varPhi(\varepsilon)=\int_{-\varepsilon}^{\varepsilon}|\overline{\hat{u}(z_1)}\cdot\hat{v}(iy)|dy.
	\end{equation*}
	Since $Re[z_1^{\alpha}]=|z_1|^{\alpha}\cos\theta $, we have
	\begin{equation*}\label{semi-discrete Stability3-16}	
		\begin{split}
			\varPhi^{2}(\varepsilon)&\leq4\int_{0}^{\infty}\dfrac{Re[(z_1)^{\alpha}]}{\cos\theta}\cdot|\hat{v}(iy)|^{2}dy\cdot\int_{0}^{\infty}\dfrac{Re[z_1^{-\alpha}]}{\cos(-\theta)}\cdot|\hat{u}(z_1)|^{2}dy\\
			&\leq4\sec^{2}\Big(\dfrac{\pi\alpha}{2}\Big)\int_{0}^{\infty}Re[z_1^{\alpha}]\cdot|\hat{v}(iy)|^{2}dy\cdot\int_{0}^{\infty}Re[z_1^{-\alpha}]\cdot|\hat{u}(z_1)|^{2}dy \\
			&=\sec^{2}\Big(\dfrac{\pi\alpha}{2}\Big)\int_{-\infty}^{\infty}z_1^{\alpha}\cdot|\hat{v}(iy)|^{2}dy\cdot\int_{-\infty}^{\infty}z_1^{-\alpha}\cdot|\hat{u}(z_1)|^{2}dy,	\\
		\end{split}
	\end{equation*}
	where $ \theta\leq\dfrac{\pi\alpha}{2} $ is the argument of a complex number $ z_1^{\alpha} $ and $ Re[z_1] $ is the real part of a complex number $ z_1 $. Next, applying Plancherel's theorem and Lemma \ref{F_T}, we have
	\begin{equation*}\label{semi-discrete Stability3-17}
		\begin{split}
			\varPhi^{2}(\varepsilon)&\leq(2\pi)^{2}\sec^{2}\Big(\dfrac{\pi\alpha}{2}\Big)\int_{0}^{T}v(t)\cdot {}_0^R D_t^{\alpha,\phi(\x)}v(t)dt\\
			&~~~~\cdot\int_{0}^{T}e^{-2\phi(\x)t}u(t)\cdot {}_0I_t^{\alpha,0}u(t)dt.
		\end{split}			
	\end{equation*}	
	
	As a result of this estimate and \eqref{semi-discrete Stability3-13}, we have
	\begin{equation*}\label{semi-discrete Stability3-18}
		\begin{split}
			\Big|\int_{0}^{T}e^{-\phi(\x)t}u(t)\cdot v(t)dt\Big|^{2}&\leq \sec^{2}\Big(\dfrac{\pi\alpha}{2}\Big)\int_{0}^{T}v(t)\cdot {}_0^R D_t^{\alpha,\phi(\x)}v(t)dt\\
			&~~~~\cdot\int_{0}^{T}e^{-2\phi(\x)t}u(t)\cdot {}_0I_t^{\alpha,0}u(t)dt.
		\end{split}
	\end{equation*}
	
	{\bf Case II}:
	$ \phi(\x)< 0 \
	\textrm{for}~ \x \in \Omega\backslash\Omega_{1}$.  By Plancherel's theorem, there exists
	\begin{equation}\label{semi-discrete Stability3-19}
		\begin{split}
			\int_{0}^{T}e^{-\phi(\x)t}u(t)\cdot v(t)dt
			&=\int_{0}^{T}e^{2\phi(\x)t}v(t)\cdot e^{-3\phi(\x)t}u(t)dt\\
			&\leq\dfrac{1}{2\pi}\int_{-\infty}^{\infty}|\overline{\hat{v}(iy-2\phi(\x))}\cdot\hat{u}(iy+3\phi(\x))| dy. \\
		\end{split}
	\end{equation}
	
	Similar to \eqref{semi-discrete Stability3-17}, for any $ \varepsilon>0 $, we have
	\begin{equation*}\label{semi-discrete Stability3-20}
		\begin{split}
			\Psi^{2}(\varepsilon)
			&\leq\sec^{2}\Big(\dfrac{\pi\alpha}{2}\Big)\int_{-\infty}^{\infty}z_2^{\alpha}\cdot|\hat{v}(z_2-\phi(\x))|^{2}dy\cdot\int_{-\infty}^{\infty}z_2^{-\alpha}\cdot|\hat{u}(iy+3\phi(\x))|^{2}dy	\\
			&=\sec^{2}\Big(\dfrac{\pi\alpha}{2}\Big)\int_{-\infty}^{\infty}z_1^{\alpha}\cdot|\hat{v}(iy)|^{2}dy\cdot\int_{-\infty}^{\infty}z_2^{-\alpha}\cdot|\hat{u}(iy+3\phi(\x))|^{2}dy,	\\
		\end{split}
	\end{equation*}
	where
	\begin{equation*}\label{semi-discrete Stability3-21}
		\begin{aligned}
			& z_1=iy+\phi(\x),\ z_2=iy-\phi(\x),\\
			&\Psi(\varepsilon)=\int_{-\varepsilon}^{\varepsilon}|\overline{\hat{v}(iy-2\phi(\x))}\cdot\hat{u}(iy+3\phi(\x))|dy,\\
		\end{aligned}
	\end{equation*}
	and similar arguments to prove \eqref{semi-discrete Stability3-1-10} are used. Further applying Plancherel's theorem, we have
	\begin{equation*}\label{semi-discrete Stability3-22}
		\begin{split}
			\Psi^{2}(\varepsilon) &\leq(2\pi)^{2}\sec^{2}\Big(\dfrac{\pi\alpha}{2}\Big)\int_{0}^{T} v(t)\cdot {}_0^R D_t^{\alpha,\phi(\x)}v(t)dt\\
			&~~~~\cdot\int_{0}^{T}e^{-3\phi(\x)t}u(t)\cdot e^{-3\phi(\x)t}{}_0I_t^{\alpha,-4\phi(\x)}u(t)dt.
		\end{split}
	\end{equation*}
	
	As a result of this estimate and \eqref{semi-discrete Stability3-19}, we have
	\begin{equation*}\label{semi-discrete Stability3-23}
		\begin{split}
			\Big|\int_{0}^{T}e^{-\phi(\x)t}u(t)\cdot v(t)dt\Big|^{2}&\leq \sec^{2}\Big(\dfrac{\pi\alpha}{2}\Big)\int_{0}^{T}  v(t)\cdot {}_0^R D_t^{\alpha,\phi(\x)}v(t)dt\\
			&~~~~\cdot\int_{0}^{T}e^{-6\phi(\x)t}u(t)\cdot {}_0I_t^{\alpha,-4\phi(\x)}u(t)dt.
		\end{split}
	\end{equation*}
	
	Thus the proof is completed.
\end{proof}

\subsection{Notations and function spaces}
\label{Sec2_2}

Now, let's introduce the symbols to be used later. Let $ \Omega_{h}= \{\Omega_{ij}\}_{i=1,\cdots,N_{x}}^{j=1,\cdots,N_{y}} $ denote a tessellation of $ \Omega $ with rectangular element $ \Omega_{ij}=I_{i}\times J_{j}$, where $ I_{i}=(x_{i-\frac{1}{2}},x_{i+\frac{1}{2}}) $ and  $ J_{j}=(y_{j-\frac{1}{2}},y_{j+\frac{1}{2}}) $ with the length $ h_{i}^{x}=x_{i+\frac{1}{2}}-x_{i-\frac{1}{2}} $ and width $ h_{j}^{y}=y_{j+\frac{1}{2}}-y_{j-\frac{1}{2}} $. Set $ h_{ij}=\max(h_{i}^{x},h_{i}^{y}) $ and denote $h=\max\limits_{\Omega_{ij}\in\Omega_{h}}h _{ij}$. We also
assume that $ \Omega_{h} $ is quasi-uniform in this paper.
Moreover, we define %$ L^{2}(\Omega_{h}) $ is defined as
\begin{equation*}\label{Notations2-1}
	L^{2}(\Omega_{h}):=\{v:\Omega\longrightarrow \mathbb{R}\big|\ v|_{\Omega_{ij}}\in L^{2}(\Omega_{ij}), \, \forall\, \Omega_{ij}\in\Omega_{h}\}
\end{equation*}
and the finite element space consisting of piecewise polynomials
\begin{equation*}\label{Notations1}
	V_{h}^{k}=\{v\in L^{2}(\Omega_{h})\big|\ v|_{\Omega_{ij}}\in Q^{k}(\Omega_{ij}),\, \forall\,\Omega_{ij}\in\Omega_{h}\},
\end{equation*}
\begin{equation*}\label{Notations1_1}
	\textit{\textbf{V}}_{h}^{k}=\{(v,w)\big|v,w\in V_{h}^{k}\},
\end{equation*}
where $Q^{k}(\Omega_{ij})=P^{k}(I_{i})\otimes P^{k}(J_{j})$ with  $\otimes$ being the tensor product. Here $P^{k}(I_{i})$ and $P^{k}(J_{j})$ denote the set of all polynomials of degrees at most $ k $ on edges $I_{i}$ and $J_{j}$, respectively.

The broken Sobolev space $ H^{s}(\Omega_{h}) $, for any given integer $s\geq0$, is defined as
\begin{equation*}\label{Notations2-2}
	H^{s}(\Omega_{h}):=\{v:\Omega\longrightarrow \mathbb{R}\big|\ v|_{\Omega_{ij}}\in H^{s}(\Omega_{ij}),\,  \forall\,\Omega_{ij}\in\Omega_{h}\},
\end{equation*}
\begin{equation*}\label{Notations2-3}
	\textit{\textbf{H}}^{s}(\Omega_{h})=\{(v,w)\big|v,w\in 	H^{s}(\Omega_{h})\},
\end{equation*}
equipped with the broken Sobolev norm
\begin{equation*}\label{Notations4_3}
	\|v\|_{H^{s}(\Omega_{h})}=\Big(\sum_{\Omega_{ij}\in\Omega_{h}}\|v\|_{H^{s}(\Omega_{ij})}^{2}\Big)^{1/2},\quad  \|v\|_{\textbf{\textit{H}}^{s}(\Omega_{h})}=\Big(\sum_{\Omega_{ij}\in\Omega_{h}}\|v\|_{\textbf{\textit{H}}^{s}(\Omega_{ij})}^{2}\Big)^{1/2}.
\end{equation*}
%	We set $L^2$-norm in element $\Omega_{ij}$
%	\begin{equation*}\label{Notations3}
%		\begin{aligned}
%		&(v,r)_{\Omega_{ij}}=\int_{\Omega_{ij}}v(\x)r(\x)d\x ,\quad \|v\|_{\Omega_{ij}}=\sqrt{(v,v)_{\Omega_{ij}}},\\
%		&(\textit{\textbf{v}},\textit{\textbf{r}}) _{\Omega_{ij}}=\int_{\Omega_{ij}}\textit{\textbf{v}}(\x)\cdot\textit{\textbf{r}}(\x)d\x ,\quad\|\textit{\textbf{v}}\|_{\Omega_{ij}}=\sqrt{(\textit{\textbf{v}},\textit{\textbf{v}})_{\Omega_{ij}}}, \\
%		&	\langle v,\textit{\textbf{v}}\cdot\textit{\textbf{n}}\rangle _{\partial\Omega_{ij}}=\int_{\partial\Omega_{ij}}v(\textit{\textbf{s}})\textit{\textbf{v}}(\textit{\textbf{s}})\cdot \textit{\textbf{n}} d\textit{\textbf{s}}.\\
%	\end{aligned}
%	\end{equation*}
For any $\Omega_{ij}\in\Omega_{h}$, we write $(\cdot,\cdot)_{\Omega_{ij}}$ and $\|\cdot\|_{\Omega_{ij}}$ as the inner product and norm associated with $L^2(\Omega_{ij})$. To simplify symbols, summing over all the elements, we denote
\begin{equation*}\label{Notations4}
	\begin{aligned}
		&(v,r)=\sum_{\Omega_{ij}\in\Omega_{h}}(v,r)_{\Omega_{ij}} ,\quad \|v\|=\Big(\sum_{\Omega_{ij}\in\Omega_{h}}\|v\|_{\Omega_{ij}}^{2}\Big)^{1/2},\\
		&	(\textit{\textbf{v}},\textit{\textbf{r}})=\sum_{\Omega_{ij}\in\Omega_{h}}(\textit{\textbf{v}},\textit{\textbf{r}})_{\Omega_{ij}} ,\quad \|\textit{\textbf{v}}\|=\Big(\sum_{\Omega_{ij}\in\Omega_{h}}\|\textit{\textbf{v}}\|_{\Omega_{ij}}^{2}\Big)^{1/2},\\
		&	\langle v,\textit{\textbf{v}}\cdot\textit{\textbf{n}}\rangle=\sum_{\Omega_{ij}\in\Omega_{h}}	\langle v,\textit{\textbf{v}}\cdot\textit{\textbf{n}}\rangle _{\partial\Omega_{ij}},\quad 	\langle v,\textit{\textbf{v}}\cdot\textit{\textbf{n}}\rangle _{\partial\Omega_{ij}}=\int_{\partial\Omega_{ij}}v(\textit{\textbf{s}})\textit{\textbf{v}}(\textit{\textbf{s}})\cdot \textit{\textbf{n}} d\textit{\textbf{s}},\\
	\end{aligned}
\end{equation*}
where $ \textit{\textbf{n}} $ is the outward unit normal vector to
$ \partial\Omega_{ij} $.

As usual, we refer to the interior information of the element by a superscript ``-'' and to the exterior information by a superscript  ``+''. Let $v_{i+1/2,y}^{\pm}$ and $v_{x,j+1/2}^{\pm}$ represent the limit values
of the function $v(\x)$ at $(x_{i+1/2},y)$ and $(x,y_{j+1/2})$, respectively, i.e., $\lim_{x\rightarrow x_{i+1/2}^{\pm}}v(\x)=v_{i+1/2,y}^{\pm}$ and $\lim_{y\rightarrow y_{j+1/2}^{\pm}}v(\x)=v_{x,j+1/2}^{\pm}$.  Moreover, the weighted averages are denoted by
\begin{equation}\label{jumpaverage}
	\begin{aligned}
		&v_{i+1/2,y}^{(\sigma_1,y)}=\sigma_1 v_{i+1/2,y}^{-}+(1-\sigma_1)v_{i+1/2,y}^{+} ,\\  &v_{x,j+1/2}^{(x,\sigma_2)}=\sigma_2 v_{x,j+1/2}^{-}+(1-\sigma_2)v_{x,j+1/2}^{+},
	\end{aligned}
\end{equation}
where $\sigma_1$, $\sigma_2$ are the given weights.

\section{The spatial semi-discrete LDG scheme}
\label{Sec3}

In the section, we present the semi-discrete LDG scheme for Eq.~\eqref{equivalent}. The proof details of the $L^2$-stability  and the optimal convergence results for the semi-discrete scheme are provided. These results are helpful to the numerical analyses in Section \ref{Sec4}.

\subsection{Variational formulation and numerical scheme}
\label{Sec3_1}

 Let us introduce the auxiliary variable $\textbf{\textit{p}}$, and rewrite Eq.~\eqref{equivalent} as a first-order system
\begin{equation}\label{reequivalent}
	\left\{
	\begin{aligned}
		& {}_0^C D_t^{\alpha,\kappa(\x)}u(t,\x)-\nabla \cdot \textit{\textbf{p}}(t,\x)=0,&&(t,\x)\in (0,T]\times\Omega, \\
		& \textit{\textbf{p}}(t,\x)-\nabla u(t,\x)=0 , &&(t,\x)\in (0,T]\times\Omega,\\
		& u(0,\x)=u_{0}(\x) , &&\x\in \Omega.\\
	\end{aligned} \right.
\end{equation}
Assume that $ (u,\textit{\textbf{p}}) $ as the exact solution of \eqref{reequivalent} belongs to
\begin{equation*}\label{semi-discrete0-1}
	H^{1}(0,T; H^{1}(\Omega_{h}))\times L^{2}(0,T; \textit{\textbf{H}}^{1}(\Omega_{h})).
\end{equation*}
%The following system is obtained by \eqref{reequivalent} and taking inner product of test
%functions $ (v,\textit{\textbf{w}}) $, including integration by part technique
Taking the inner product of test
functions $ (v,\textit{\textbf{w}})$ and doing integration by part, from \eqref{reequivalent} we get
\begin{equation}\label{Variational}
	\left\{
	\begin{aligned}
		& ({}_0^C D_t^{\alpha,\kappa(\x)}u(t,\x),v)_{\Omega_{ij}}+(\textit{\textbf{p}}(t,\x),\nabla v)_{\Omega_{ij}}-\langle\textit{\textbf{p}}(t,\x)\cdot\textit{\textbf{n}}, v\rangle_{\partial\Omega_{ij}}=0  , \\
		&(\textit{\textbf{p}}(t,\x),\textit{\textbf{w}})_{\Omega_{ij}}+(u(t,\x),\nabla\cdot\textit{\textbf{w}})_{\Omega_{ij}}-\langle u(t,\x),\textit{\textbf{w}}\cdot\textit{\textbf{n}}\rangle_{\partial\Omega_{ij}}=0 ,
	\end{aligned}\right.
\end{equation}
 for $ \Omega_{ij}\in\Omega_{h}$, where $ (v,\textit{\textbf{w}})\in H^{1}(\Omega_{h})\times \textit{\textbf{H}}^{1}(\Omega_{h}) $.

To obtain the spatial semi-discrete scheme for Eq.~\eqref{equivalent}, we present the numerical fluxes $\widehat{u_{h}}$ and $\widehat{\textit{\textbf{p}}_{h}}$ as single valued functions defined at the cell interfaces, in general depending on the values of the numerical solution $u_{h}$ and $\textit{\textbf{p}}_{h}$ from both sides of the interfaces
\begin{equation*}
	\begin{aligned}
		&(\widehat{u_{h}})_{x,j+1/2}=\widehat{u_{h}}((u_h)_{x,j+1/2}^{-},(u_h)_{x,j+1/2}^{+}),\\
		&(\widehat{u_{h}})_{i+1/2,y}=\widehat{u_{h}}((u_h)_{i+1/2,y}^{-},(u_h)_{i+1/2,y}^{+}),\\
%		 &(\widehat{\textit{\textbf{p}}_h})_{x,j+1/2}=\widehat{\textit{\textbf{p}}_{h}}((\textit{\textbf{p}}_h)_{x,j+1/2}^{-},(\textit{\textbf{p}}_h)_{x,j+1/2}^{+})\\
%		&(\widehat{\textit{\textbf{p}}_h})_{i+1/2,y}=\widehat{\textit{\textbf{p}}_{h}}((\textit{\textbf{p}}_h)_{i+1/2,y}^{-},(\textit{\textbf{p}}_h)_{i+1/2,y}^{+}),\\	
		&(\widehat{\textit{\textbf{p}}_h})_{x,j+1/2}=[(\widehat{p_h})_{x,j+1/2}\quad (\widehat{q_h})_{x,j+1/2}]^{T},\\
		&(\widehat{\textit{\textbf{p}}_h})_{i+1/2,y}=[(\widehat{p_h})_{i+1/2,y}\quad (\widehat{q_h})_{i+1/2,y}]^{T},
	\end{aligned}
\end{equation*}
where $\textit{\textbf{p}}_h=[p_h\quad q_h]^{T}$.

Now, we construct the spatial semi-discrete LDG scheme of Eq.~\eqref{equivalent}. Find $ (u_{h}(t,\x),\textit{\textbf{p}}_{h}(t,\x)) \in H^{1}(0,T; V_{h}^{k})\times L^{2}(0,T; \textit{\textbf{V}}_{h}^{k})$,
which is the approximation of $ (u(t,\x),\textit{\textbf{p}}(t,\x)) $, such that
\begin{equation}\label{semi-discrete}
	\left\{
	\begin{aligned}
		& ({}_0^C D_t^{\alpha,\kappa(\x)}u_h,v)_{\Omega_{ij}}+(\textit{\textbf{p}}_h,\nabla v)_{\Omega_{ij}}-\langle\widehat{\textit{\textbf{p}}_h}\cdot\textit{\textbf{n}}, v\rangle_{\partial\Omega_{ij}}=0  , \\
		&(\textit{\textbf{p}}_h,\textit{\textbf{w}})_{\Omega_{ij}}+(u_h,\nabla\cdot\textit{\textbf{w}})_{\Omega_{ij}}-\langle \widehat{u_h},\textit{\textbf{w}}\cdot\textit{\textbf{n}}\rangle_{\partial\Omega_{ij}}=0 ,\\
		&(u_h(0),v)_{\Omega_{ij}}=(u_{0}(\x),v)_{\Omega_{ij}},
	\end{aligned}\right.
\end{equation}
where $ (v,\textit{\textbf{w}})\in V_{h}^{k}\times \textit{\textbf{V}}_{h}^{k} $ for $ \Omega_{ij}\in\Omega_{h}$.

In the following sections, we choose the generalized alternating numerical fluxes
\begin{equation}\label{fluxes_1}
	\begin{aligned}
		&(\widehat{u_{h}})_{i+1/2,y}=(u_{h})_{i+1/2,y}^{(\sigma_1,y)}, \quad (\widehat{\textit{\textbf{p}}_{h}})_{i+1/2,y}=(\textit{\textbf{p}}_{h})_{i+1/2,y}^{(1-\sigma_1,y)},\\
		&(\widehat{u_{h}})_{x,j+1/2}=(u_{h})_{x,j+1/2}^{(x,\sigma_2)}, \quad (\widehat{\textit{\textbf{p}}_{h}})_{x,j+1/2}=(\textit{\textbf{p}}_{h})_{x,j+1/2}^{(x,1-\sigma_2)},
	\end{aligned}
\end{equation}
where $\sigma_1\neq1/2$, $\sigma_2\neq1/2$, $ i=1,\cdots,N_{x},\,j=1,\cdots,N_{y} $.
\begin{remark}
	When $\sigma_1=\sigma_2=1/2$, the generalized alternating numerical fluxes are called central numerical fluxes, which can not guarantee the optimal error estimates of the numerical scheme in theory. The numerical results for central numerical fluxes are given in Section \ref{Sec5}.
	\end{remark}

\subsection{Stability analysis of the semi-discrete scheme}
\label{Sec3_2}

To ensure the validity of the numerical scheme, we need to prove the $L^{2}$-stability and convergence of the semi-discrete LDG scheme \eqref{semi-discrete}. First, we consider the bilinear form
\begin{equation}\label{semi-discrete Stability1}
	\begin{split}
		B(u_{h},v;\textit{\textbf{p}}_{h},\textit{\textbf{w}})&=(u_h,\nabla\cdot\textit{\textbf{w}})-\langle \widehat{u_h},\textit{\textbf{w}}\cdot\textit{\textbf{n}}\rangle+(\textit{\textbf{p}}_h,\nabla v)-\langle\widehat{\textit{\textbf{p}}_h}\cdot\textit{\textbf{n}}, v\rangle.
	\end{split}
\end{equation}

\begin{lemma} \label{bilinear}
	Assume that $u_{h}$ and $\textit{\textbf{p}}_{h}$ are defined in the rectangular region $\Omega_{h}$ with periodic boundary conditions and the numerical fluxes $\widehat{u_{h}}$ and  $\widehat{\textit{\textbf{p}}_{h}}$ are given in \eqref{fluxes_1}. Then
	\begin{equation}\label{semi-discrete Stability3}
		B(u_{h},u_{h};\textit{\textbf{p}}_{h},\textit{\textbf{p}}_{h})=0.
	\end{equation}
\end{lemma}
\begin{proof}
	The formula \eqref{semi-discrete Stability3} can be directly calculated by substituting \eqref{fluxes_1} into \eqref{semi-discrete Stability1} and using the periodic boundary conditions of $u_{h}(t,\x)$ and $\textit{\textbf{p}}_{h}(t,\x)$. Therefore, we omit the details.	
	
\end{proof}	

\begin{theorem}[\label{Th3_1}$L^{2}$-stability] The semi-discrete LDG scheme \eqref{semi-discrete} %with initial boundary conditions in \eqref{Variational}
	is unconditionally stable, i.e.,
	\begin{equation}\label{semi-discrete Stability4}
		\int_{0}^{T}\|u_{h}(t,\x)\|^{2}dt\leq C\|u_{h}(0,\x)\|^{2},
	\end{equation}
	where $ C=\sec^{2}\Big(\dfrac{\pi\alpha}{2}\Big)\cdot \dfrac{T\max{\{C_{max}^{6},1\}}}{C_{\alpha}\Gamma(2-\alpha)} $. Here  $C_{max} $ and $C_{\alpha}$ are defined in \eqref{assumptions2} and Lemma \ref{L2-1}, respectively.
\end{theorem}
\begin{proof}
	Setting the test functions $v=u_{h}$ and  $\textit{\textbf{w}}=\textit{\textbf{p}}_{h}$, and summing up \eqref{semi-discrete} over all $j$ yield
	\begin{equation*}\label{semi-discrete Stability5}
		({}_0^C D_t^{\alpha,\kappa(\x)}u_{h},u_{h})+B(u_{h},u_{h};\textit{\textbf{p}}_{h},\textit{\textbf{p}}_{h})+\|\textit{\textbf{p}}_{h}\|^{2}=0,
	\end{equation*}
	where the bilinear form $B(\cdot,\cdot;\cdot,\cdot)$ is defined in \eqref{semi-discrete Stability1}.
Employing Lemma \ref{bilinear} leads to
	\begin{equation}\label{semi-discrete Stability6}
		({}_0^C D_t^{\alpha,\kappa(\x)}u_{h},u_{h})+\|\textit{\textbf{p}}_{h}\|^{2}=0.
	\end{equation}
	
	Since $ {}_0^C D_t^{\alpha}u={}_0^{R} D_t^{\alpha}u-\omega_{1-\alpha}(t)u(0) $  and  \eqref{semi-discrete Stability6} , one has
	\begin{equation}\label{semi-discrete Stability6-1}
		\begin{split}
			\int_{0}^{T}({}_0^C D_t^{\alpha,\kappa(\x)}u_{h},u_{h})dt&=\int_{0}^{T}({}_0^{R} D_t^{\alpha,\kappa(\x)}u_{h},u_{h})dt\\
			&~~~~-\int_{0}^{T}(\omega_{1-\alpha}(t)e^{-\kappa(\x)t}u_{h}(0,\x),u_{h})dt	\\
			&\leq 0,
		\end{split}
	\end{equation}
	where $ \omega_{1-\alpha}(t)=t^{-\alpha}/\Gamma(1-\alpha)$. Denote the first and second line of the right hand side of \eqref{semi-discrete Stability6-1} by $\uppercase\expandafter{\romannumeral1}$ and $\uppercase\expandafter{\romannumeral2}$, respectively. The analyses to $\uppercase\expandafter{\romannumeral1}$ and $\uppercase\expandafter{\romannumeral2}$ rely on Lemma\,\ref{L2-1} and Lemma\,\ref{L2-2}. Taking the functions $ v(t,\x)=u_{h} $ and $u(t,\x)=\omega_{1-\alpha}(t)u_{h}(0,\x) $ in Lemma\,\ref{L2-2} results in
	\begin{equation}\label{semi-discrete Stability6-2}
		\begin{split}
			\uppercase\expandafter{\romannumeral2}
			&\leq\sec\Big(\dfrac{\pi\alpha}{2}\Big)\uppercase\expandafter{\romannumeral1}^{1/2}\cdot\Big(\int_{0}^{T}\int_{\Omega}e^{-2\kappa(\x)t}\omega_{1-\alpha}(t)u_{h}(0,\x)\\
			&~~~~\cdot {}_{0}I_{t}^{\alpha}(e^{2(|\kappa(\x)|-\kappa(\x))t}\omega_{1-\alpha}(t)u_{h}(0,\x))d\x dt\Big)^{1/2}\\
			&\leq\sec\Big(\dfrac{\pi\alpha}{2}\Big)\uppercase\expandafter{\romannumeral1}^{1/2}\cdot \max{\{C_{max}^{3},1\}} \Big(\dfrac{T^{1-\alpha}}{\Gamma(2-\alpha)}\Big)^{1/2}\|u_{h}(0,\x)\|\\
			&\leq\dfrac{1}{2}\uppercase\expandafter{\romannumeral1}+\frac{1}{2}\sec^{2}\Big(\frac{\pi\alpha}{2}\Big)\cdot \max{\{C_{max}^{6},1\}}\dfrac{T^{1-\alpha}}{\Gamma(2-\alpha)}\|u_{h}(0,\x)\|^{2},
		\end{split}	
	\end{equation}
	where $ C_{\kappa} $, $ C_{max} $ are defined by \eqref{assumptions1}, \eqref{assumptions2} and the fact that ${}_0I_t^{\alpha}(\omega_{1-\alpha}(t))=1$ is used.

	Combining the estimates \eqref{semi-discrete Stability6-1} and \eqref{semi-discrete Stability6-2} leads to
	\begin{equation*}\label{semi-discrete Stability6-4}
		\uppercase\expandafter{\romannumeral1}\leq\sec^{2}\Big(\dfrac{\pi\alpha}{2}\Big)\cdot \max{\{C_{max}^{6},1\}}\dfrac{T^{1-\alpha}}{\Gamma(2-\alpha)}\|u_{h}(0,\x)\|^{2}.
	\end{equation*}
		In addition, taking the function $u(t,\x)=u_{h}(t,\x)$ in Lemma\,\ref{L2-1}, there exists 
	\begin{equation*}\label{semi-discrete Stability6-5}
		\uppercase\expandafter{\romannumeral1}\geq C_{\alpha}T^{-\alpha}\int_{0}^{T}\|u_{h}(t,\x)\|^{2}dt.
	\end{equation*}
		As a result of these estimates, one can obtain
	\begin{equation*}\label{semi-discrete Stability6-6}
		\int_{0}^{T}\|u_{h}(t,\x)\|^{2}dt\leq C\|u_{h}(0,\x)\|^{2},
	\end{equation*}
which finishes the proof of the stability result.
\end{proof}

\subsection{Error estimates of the semi-discrete scheme}
\label{Sec3_3}
 %Whether the exact solution  can be effectively approximated by the numerical solution is the core issue of our concern. Therefore, we must also analyze the error estimates of the numerical scheme in theory.
In the following, we will present error estimates for the LDG semi-discrete scheme \eqref{semi-discrete}. The optimal convergence results of order $k + 1$ are lost if the projection error at all element boundaries cannot be treated in a nice way. In light of this point, the optimal error estimates are obtained in the forthcoming error analysis by using the approximation properties of the so-called Gauss-Radau projection. The generalized Gauss-Radau projection \cite{Cheng-2017-86} $P_{\sigma_1,\sigma_2}:L^2(\Omega_h)\rightarrow V_{h}^{k}$  of the scalar function is defined by

\begin{equation*}\label{semi-discrete error1}
	\begin{aligned}
		&(P_{\sigma_1,\sigma_2}u,v)_{\Omega_{ij}}=(u,v)_{\Omega_{ij}},\quad\quad\quad\quad\quad\quad\quad~ \forall v\in Q^{k-1}(\Omega_{ij}),\\	&\Big((P_{\sigma_1,\sigma_2}u)^{(\sigma_1,y)}_{i+1/2,y},v\Big)_{J_{j}}=\Big(u^{(\sigma_1,y)}_{i+1/2,y},v\Big)_{J_{j}},\quad \forall v\in P^{k-1}(J_{j}),\\	&\Big((P_{\sigma_1,\sigma_2}u)^{(x,\sigma_2)}_{x,j+1/2},v\Big)_{I_{i}}=\Big(u^{(x, \sigma_2)}_{x,j+1/2},v\Big)_{I_{i}},\quad \forall v\in P^{k-1}(I_{i}),\\
		&(P_{\sigma_1,\sigma_2}u)^{(\sigma_1,\sigma_2)}_{i+1/2,j+1/2}=u^{(\sigma_1,\sigma_2)}_{i+1/2,j+1/2},\\
	\end{aligned}
\end{equation*}
for any $i=1,\cdots,N_{x}$ and $j=1,\cdots,N_{y}$. Here and in what follows, we use the
notation defined in \eqref{jumpaverage} to represent the weighted average on each edge and use
\begin{equation*}
	\begin{split}
		u^{(\sigma_1,\sigma_2)}_{i+1/2,j+1/2}&=\sigma_1\sigma_2u(x_{i+1/2}^{-},y_{j+1/2}^{-})+(1-\sigma_1)(1-\sigma_2)u(x_{i+1/2}^{+},y_{j+1/2}^{+})\\
		&~~~+\sigma_1(1-\sigma_2)u(x_{i+1/2}^{-},y_{j+1/2}^{+})+(1-\sigma_1)\sigma_2u(x_{i+1/2}^{+},y_{j+1/2}^{-}),
	\end{split}
\end{equation*}
to represent the weighted average at the corner point.

In order to deal with auxiliary variable, we propose the following generalized Gauss-Radau projection \cite{Cheng-2017-86} $Q_{\sigma_1}$, $Q_{\sigma_2}:H^1(\Omega_h)\rightarrow V_{h}^{k}$ of scalar function
\begin{equation}\label{semi-discrete error1_1}
	\begin{aligned}
		&(Q_{\sigma_1}u,v)_{\Omega_{ij}}=(u,v)_{\Omega_{ij}},\quad \forall v\in P^{k-1}(I_{i})\otimes P^{k}(J_{j}),\\
		&(Q_{\sigma_2}u,v)_{\Omega_{ij}}=(u,v)_{\Omega_{ij}},\quad \forall v\in P^{k}(I_{i})\otimes P^{k-1}(J_{j}),\\
		&\Big((Q_{\sigma_1}u)^{(\sigma_1,y)}_{i+1/2,y},v\Big)_{J_{j}}=\Big(u^{(\sigma_1,y)}_{i+1/2,y},v\Big)_{J_{j}},\quad \forall v\in P^{k}(J_{j}),\\
		&\Big((Q_{\sigma_2}u)^{(x,\sigma_2)}_{x,j+1/2},v\Big)_{I_{i}}=\Big(u^{(x,\sigma_2)}_{x,j+1/2},v\Big)_{I_{i}},\quad \forall v\in P^{k}(I_{i}),\\		
	\end{aligned}
\end{equation}
for any $i=1,\cdots,N_{x}$ and $j=1,\cdots,N_{y}$. Moreover, the generalized Gauss-Radau projection $Q_{\sigma_1,\sigma_2}:\textbf{\textit{H}}^1(\Omega_h)\rightarrow \textbf{\textit{V}}_{h}^{k}$ of vector function can be defined as $Q_{\sigma_1,\sigma_2}\textit{\textbf{u}}=[Q_{\sigma_1}u_1\quad Q_{\sigma_2}u_2]^{T}$, where $\textit{\textbf{u}}=[u_1\quad u_2]^{T}$. Then, the following three approximation properties hold \cite{Cheng-2017-86}.
\begin{lemma}\label{projection_err_u}
	Assume that $u(\x)\in H^{s+1}(\Omega_{h})\cap H^{2}(\Omega_{h})$ with $ \Omega_h \subset \mathbb{R}^2$ and $ s\geq0 $. If $\sigma_1\neq1/2$ and $\sigma_2\neq1/2$, the projections $P_{\sigma_1,\sigma_2}$ satisfy
	\begin{equation*}\label{semi-discrete error3}
		\begin{aligned}	
			&\|\eta\|+h^{1/2}\|\eta\|_{L^{2}(\Gamma_{h})}\leq Ch^{\min(k+1,s+1)}\|u\|_{H^{s+1}(\Omega_{h})},\\
		\end{aligned}
	\end{equation*}
	where $\eta=P_{\sigma_1,\sigma_2}u(\x)-u(\x)$. The positive constant $C$ is independent of $u(\x)$ and $h$. Here $\Gamma_{h}$ denotes the set of boundary points of all elements $ \Omega_{ij}\in\Omega_{h}$.
\end{lemma}

\begin{lemma} \label{projection_err_p}
	Assume that $u(\x)\in H^{s+1}(\Omega_{h})$ with $ \Omega_h \subset \mathbb{R}^2$ and $ s\geq0 $. If $\sigma_l\neq1/2$, the projections $Q_{\sigma_l}$ satisfy
	\begin{equation*}\label{semi-discrete error3_1}
		\begin{aligned}	
			&\|\eta_l\|+h^{1/2}\|\eta_l\|_{L^{2}(\Gamma_{h})}\leq Ch^{min(k+1,s+1)}\|u\|_{H^{s+1}(\Omega_{h})}, \quad l=1,2,
		\end{aligned}
	\end{equation*}
	where $\eta_l=Q_{\sigma_l}u(\x)-u(\x)$. The positive constant $C$ is independent of $u(\x)$ and $h$. Here $\Gamma_{h}$ denotes the set of boundary points of all elements $ \Omega_{ij}\in\Omega_{h}$.
\end{lemma}

\begin{lemma} \label{vector_projection}
	Assume that $u(\x)\in H^{k+2}(\Omega_{h})$ with $ \Omega_h \subset \mathbb{R}^2$. If $\sigma_1\neq1/2$ and $\sigma_2\neq1/2$, the projections $P_{\sigma_1,\sigma_2}$ satisfy
	\begin{equation}\label{semi-discrete error3_2}
		\begin{aligned}	
			&|(\eta,\nabla\cdot\textbf{v})-\langle \widehat{\eta},  \textbf{v}\cdot\textit{\textbf{n}}\rangle|\leq Ch^{k+1}\|u\|_{H^{k+2}(\Omega_{h})}\|\textbf{v}\|,\quad \forall\, \textbf{v}\in 	\textit{\textbf{V}}_{h}^{k}, \\
		\end{aligned}
	\end{equation}
	where $\eta=P_{\sigma_1,\sigma_2}u(\x)-u(\x)$ and $ \textit{\textbf{n}} $ is the outward unit normal vector to
	$ \partial\Omega_{ij} $. The positive constant $C$ is independent of $u(\x)$ and $h$. The definition of $\widehat{\eta}$ is the same as that of $\widehat{u_h}$ in \eqref{fluxes_1}.
	
	%	 $(\widehat{\eta})_{i+1/2,y}=(\eta)_{i+1/2,y}^{(\sigma_1,y)}$, $(\widehat{\eta})_{x,j+1/2}=(\eta)_{x,j+1/2}^{(x,\sigma_2)}$, for $i=1,\cdots,N_{x}$, $j=1,\cdots,N_{y}$.
\end{lemma}
\begin{proof}
Following \cite[Lemma 3.6]{Cheng-2017-86}, it is easy to get \eqref{semi-discrete error3_2}.
	%\eqref{semi-discrete error3_2} is easy to be proved according to \cite[Lemma 3.6]{Cheng-2017-86}.
\end{proof}

With the above preliminary knowledge, we next numerically analyze the convergence of the semi-discrete scheme.

\begin{theorem}\label{Th3_2}
	Let $u_{h}$ be the numerical solution of the semi-discrete LDG scheme \eqref{semi-discrete} and $u(t,\x) \in H^{1}((0,T];H^{s+1}(\Omega_{h}) \cap H^{2}(\Omega_{h}))$ the exact solution of Eq.~\eqref{equivalent} with $s\geq k $. Then there exists a positive constant $C$, being independent of $h$ and $u_{h}$, such that
	\begin{equation*}\label{semi-discrete error4}
		\int_{0}^{T}\|u(t,\x)-u_{h}(t,\x)\|^{2}dt\leq Ch^{2k+2}.
	\end{equation*}
\end{theorem}
\begin{proof}
	%$u(t,\x)$ is the exact solution of Eq.~\eqref{equivalent} which satisfies Eq.~\eqref{Variational}. 
Let $(u,\textit{\textbf{p}})$ be the exact solution of Eq.~\eqref{Variational} and $(u_{h},\textit{\textbf{p}}_{h})$ the numerical solution of the scheme \eqref{semi-discrete}. Summing up \eqref{Variational} and \eqref{semi-discrete} over all $j$ leads to
	\begin{equation}\label{semi-discrete error5}
		\Big({}_0^C D_t^{\alpha,\kappa(\x)}u_{h},v\Big)+B(u_{h},v;\textit{\textbf{p}}_{h},\textit{\textbf{w}})+(\textit{\textbf{p}}_{h},\textit{\textbf{w}})=0,
	\end{equation}
	and
	\begin{equation}\label{semi-discrete error6}
		\Big({}_0^C D_t^{\alpha,\kappa(\x)}u,v\Big)+B(u,v;\textit{\textbf{p}},\textit{\textbf{w}})+(\textit{\textbf{p}},\textit{\textbf{w}})=0,
	\end{equation}
	where the bilinear form $B(\cdot,\cdot;\cdot,\cdot)$ is defined in \eqref{semi-discrete Stability1}.
	Subtracting \eqref{semi-discrete error5} from \eqref{semi-discrete error6} gets the error equation
	\begin{equation*}\label{semi-discrete error7}
		({}_0^C D_t^{\alpha,\kappa(\x)}e_{u},v)+B(e_{u},v;e_{\textit{\textbf{p}}},\textit{\textbf{w}})+(e_{\textit{\textbf{p}}},\textit{\textbf{w}})=0,
	\end{equation*}
	where $e_{u}=u-u_{h}$, $e_{\textit{\textbf{p}}}=\textit{\textbf{p}}-\textit{\textbf{p}}_{h}$.
	
	By convention, let
	\begin{equation*}\label{semi-discrete error8}
		\begin{aligned}
			&\zeta_{u}=P_{\sigma_1,\sigma_2}e_{u}, \quad \eta_{u}=P_{\sigma_1,\sigma_2}u-u, \\
			& \zeta_{\textit{\textbf{p}}}=Q_{1-\sigma_1,1-\sigma_2}e_{\textit{\textbf{p}}}, \quad \eta_{\textit{\textbf{p}}}=Q_{1-\sigma_1,1-\sigma_2}\textit{\textbf{p}}-\textit{\textbf{p}}.
		\end{aligned}
	\end{equation*}
	Then the errors can be divided into
	$e_{u}=\zeta_{u}-\eta_{u}$ and $e_{\textit{\textbf{p}}}=\zeta_{\textit{\textbf{p}}}-\eta_{\textit{\textbf{p}}}$, which implies that
	\begin{equation*}\label{semi-discrete error9}
		\begin{split}
			\Big({}_0^C D_t^{\alpha,\kappa(\x)}\zeta_{u},v\Big)+(\zeta_{\textit{\textbf{p}}},\textit{\textbf{w}})&=\Big({}_0^C D_t^{\alpha,\kappa(\x)}\eta_{u},v\Big)+(\eta_{\textit{\textbf{p}}},\textit{\textbf{w}})\\
			&~~~~+B(\eta_{u},v;\eta_{\textit{\textbf{p}}},\textit{\textbf{w}})-B(\zeta_{u},v;\zeta_{\textit{\textbf{p}}},\textit{\textbf{w}}).\\
		\end{split}
	\end{equation*}
Setting the test functions $v=\zeta_{u}$, $\textit{\textbf{w}}=\zeta_{\textit{\textbf{p}}}$ in the above equation leads to
	\begin{equation*}\label{semi-discrete error9_1}
		\begin{split}
			\Big({}_0^C D_t^{\alpha,\kappa(\x)}\zeta_{u},\zeta_{u}\Big)+(\zeta_{\textit{\textbf{p}}},\zeta_{\textit{\textbf{p}}})&=\Big({}_0^C D_t^{\alpha,\kappa(\x)}\eta_{u},\zeta_{u}\Big)+(\eta_{\textit{\textbf{p}}},\zeta_{\textit{\textbf{p}}})\\
			&~~~~+B(\eta_{u},\zeta_{u};\eta_{\textit{\textbf{p}}},\zeta_{\textit{\textbf{p}}})-B(\zeta_{u},\zeta_{u};\zeta_{\textit{\textbf{p}}},\zeta_{\textit{\textbf{p}}}).\\
		\end{split}
	\end{equation*}
	
	%		\begin{equation*}\label{semi-discrete error10}
	%			\begin{split}
	%				B(\eta_{u},\zeta_{u};\eta_{\textit{\textbf{p}}},\zeta_{\textit{\textbf{p}}})&=(\eta_{u},\nabla\cdot\zeta_{\textit{\textbf{p}}})-\langle \widehat{\eta_{u}},\zeta_{\textit{\textbf{p}}}\cdot\textit{\textbf{n}}\rangle+(\eta_{\textit{\textbf{p}}},\nabla \zeta_{u})-\langle\widehat{\eta_{\textit{\textbf{p}}}}\cdot\textit{\textbf{n}},\zeta_{u}\rangle.
	%			\end{split}
	%		\end{equation*}

	Next, we analyze the two bilinear forms in the above equation. By using \eqref{fluxes_1}, \eqref{semi-discrete error1_1}, and Lemma\,\ref{vector_projection}, we have
	\begin{equation*}\label{semi-discrete error11}
		\begin{aligned}
			&(\eta_{\textit{\textbf{p}}},\nabla \zeta_{u})= \langle\widehat{\eta_{\textit{\textbf{p}}}}\cdot\textit{\textbf{n}},\zeta_{u}\rangle=0,\\
			&|(\eta_{u},\nabla\cdot\zeta_{\textit{\textbf{p}}})-\langle \widehat{\eta_{u}},\zeta_{\textit{\textbf{p}}}\cdot\textit{\textbf{n}}\rangle|\leq Ch^{k+1}\|\zeta_{\textit{\textbf{p}}}\|,\\
		\end{aligned}
	\end{equation*}	
	and as a result,
	\begin{equation*}\label{semi-discrete error12}
		|B(\eta_{u},\zeta_{u};\eta_{\textit{\textbf{p}}},\zeta_{\textit{\textbf{p}}})|\leq Ch^{k+1}\|\zeta_{\textit{\textbf{p}}}\|.
	\end{equation*}
From Lemma\,\ref{bilinear}, there exists $B(\zeta_{u},\zeta_{u};\zeta_{\textit{\textbf{p}}},\zeta_{\textit{\textbf{p}}})=0$. Then, one can get
	\begin{equation}\label{semi-discrete error14}
		\Big({}_0^C D_t^{\alpha,\kappa(\x)}\zeta_{u},\zeta_{u}\Big)+\|\zeta_{\textit{\textbf{p}}}\|^{2}\leq\Big({}_0^C D_t^{\alpha,\kappa(\x)}\eta_{u},\zeta_{u}\Big)+(\eta_{\textit{\textbf{p}}},\zeta_{\textit{\textbf{p}}})+Ch^{k+1}\|\zeta_{\textit{\textbf{p}}}\|.
	\end{equation}
Now, multiplying $ e^{-2C_{\kappa}t} $ on both sides of the inequality \eqref{semi-discrete error14}, we can do the estimate further for the right-hand side of \eqref{semi-discrete error14} by employing the Cauchy-Schwarz inequality
	\begin{equation}\label{semi-discrete error14-1}
		\Big({}_0^C D_t^{\alpha,\kappa(\x)}\zeta_{u},e^{-2C_{\kappa}t}\zeta_{u}\Big)\leq\Big({}_0^C D_t^{\alpha,\kappa(\x)}\eta_{u},e^{-2C_{\kappa}t}\zeta_{u}\Big)+\frac{1}{2}\|\eta_{\textit{\textbf{p}}}\|^{2}+Ch^{2k+2}.
	\end{equation}
	Similar to \eqref{semi-discrete Stability6-1}, there exist
	\begin{equation}\label{semi-discrete error15-1}
		\int_{0}^{T}\Big({}_0^C D_t^{\alpha,\kappa(\x)}\zeta_{u},e^{-2C_{\kappa}t}\zeta_{u}\Big)dt=\int_{0}^{T}\Big({}_0^R D_t^{\alpha,\kappa(\x)}\zeta_{u},e^{-2C_{\kappa}t}\zeta_{u}\Big)dt
	\end{equation}
	and
	\begin{equation}\label{semi-discrete error15-2}
		\begin{split}
			\int_{0}^{T}\Big({}_0^C D_t^{\alpha,\kappa(\x)}\eta_{u},e^{-2C_{\kappa}t}\zeta_{u}\Big)dt&=\int_{0}^{T}\Big({}_0^RD_t^{\alpha,\kappa(\x)}\eta_{u},e^{-2C_{\kappa}t}\zeta_{u}\Big)dt\\
			&~~~-\int_{0}^{T}\Big(\omega_{1-\alpha}(t)\eta_{u}(0,\x),\zeta_{u}e^{-(\kappa(\x)+2C_{\kappa})t}\Big)dt	,\\
		\end{split}
	\end{equation}
	where $ \zeta_{u}(0,\x)=0 $ is used.
	
	Denote each line of the right-hand sides of \eqref{semi-discrete error15-1} and \eqref{semi-discrete error15-2} by $\uppercase\expandafter{\romannumeral3} $, $ \uppercase\expandafter{\romannumeral4} $, and $ \uppercase\expandafter{\romannumeral5}$, respectively. The analyses of $\uppercase\expandafter{\romannumeral3}$- $\uppercase\expandafter{\romannumeral5}$ mainly depend on Lemma\,\ref{L2-1} and Lemma\,\ref{L2-2}. Setting functions $ v(t,\x)=e^{-C_{\kappa}t}\zeta_{u} $, $u(t,\x)=e^{\kappa(\x)t}{}_0^{R} D_t^{\alpha,\kappa(\x)}\eta_{u} $, and $ \phi(\x)=\kappa(\x)+C_{\kappa}>0$ in  Lemma\,\ref{L2-2}, one can get
	\begin{equation}\label{semi-discrete error15-3}
		\uppercase\expandafter{\romannumeral4}
		\leq\sec\Big(\dfrac{\pi\alpha}{2}\Big)\uppercase\expandafter{\romannumeral3} ^{1/2}\cdot \uppercase\expandafter{\romannumeral6} ^{1/2}\leq\dfrac{1}{4}\uppercase\expandafter{\romannumeral3}+\sec^{2}\Big(\dfrac{\pi\alpha}{2}\Big)\cdot \uppercase\expandafter{\romannumeral6},
	\end{equation}
	where
	\begin{equation*}\label{semi-discrete error15-3_1}
		\uppercase\expandafter{\romannumeral6}=\int_{0}^{T}\Big({}_0^{R} D_t^{\alpha,\kappa(\x)}\eta_{u},e^{-2C_{\kappa}t}\eta_{u}\Big)dt.
	\end{equation*}
	Similarly, taking functions $ v(t,\x)=e^{-C_{\kappa}t}\zeta_{u} $, $u(t,\x)=\omega_{1-\alpha}(t)\eta_{u}(0,\x)  $, and $ \phi(\x)=\kappa(\x)+C_{\kappa}>0$ in Lemma\,\ref{L2-2}, one can obtain
	\begin{equation}\label{semi-discrete error15-4}
		\begin{split}
			\uppercase\expandafter{\romannumeral5}
			&\leq\sec\Big(\dfrac{\pi\alpha}{2}\Big)\uppercase\expandafter{\romannumeral3}^{1/2}\cdot\Big(\int_{0}^{T}\omega_{1-\alpha}(t)\cdot\|e^{-(\kappa(\x)+C_{\kappa})t}\eta_{u}(0,\x)\|^{2}dt\Big)^{1/2}\\
			&\leq\dfrac{1}{4}\uppercase\expandafter{\romannumeral3}+\sec^{2}\Big(\dfrac{\pi\alpha}{2}\Big)\cdot\dfrac{T^{1-\alpha}}{\Gamma(2-\alpha)}\|\eta_{u}(0,\x)\|^{2}.\\
		\end{split}	
	\end{equation}
Substituting \eqref{semi-discrete error15-1}-\eqref{semi-discrete error15-4} into \eqref{semi-discrete error14-1} leads to
		\begin{equation}\label{semi-discrete error15-5}
		\uppercase\expandafter{\romannumeral3}\leq2\sec^{2}\Big(\dfrac{\pi\alpha}{2}\Big)\cdot\Big(\uppercase\expandafter{\romannumeral6}+\dfrac{T^{1-\alpha}}{\Gamma(2-\alpha)}\|\eta_{u}(0,\x)\|^{2}\Big)+\int_{0}^{T}\|\eta_{\textit{\textbf{p}}}\|^{2}dt+Ch^{2k+2}.
	\end{equation}
In addition, using the properties of projections in Lemma\,\ref{projection_err_u} and Lemma\,\ref{projection_err_p} yields
	\begin{equation*}\label{semi-discrete error15-5-1}
		\uppercase\expandafter{\romannumeral3}\leq 2\sec^{2}\Big(\dfrac{\pi\alpha}{2}\Big)\uppercase\expandafter{\romannumeral6}+Ch^{2k+2}.
	\end{equation*}
	
	In what follows, we will analyze $ \uppercase\expandafter{\romannumeral6} $.
	By integrating $ \uppercase\expandafter{\romannumeral6} $ by parts with respect to $ t $, one can get
	\begin{equation*}\label{semi-discrete error15-6}
		\begin{split}
			\uppercase\expandafter{\romannumeral6}&=\int_{0}^{T}\frac{(T-\tau)^{-\alpha}}{\Gamma(1-\alpha)}\Big(e^{-(\kappa(\x)+2C_{\kappa})T} \eta_{u}(T,\x),e^{\kappa(\x)\tau}\eta_{u}(\tau,\x)  \Big)d\tau\\
			&~~~~-\int_{0}^{T}\int_{0}^{t}\frac{(t-\tau)^{-\alpha}}{\Gamma(1-\alpha)}\Big(e^{\kappa(\x)\tau}\eta_{u}(\tau,\x) ,\partial_{t}\Big(e^{-(\kappa(\x)+2C_{\kappa})t} \eta_{u}(t,\x)\Big) \Big)d\tau dt.\\
%			&\leq\int_{0}^{T}\frac{(T-\tau)^{-\alpha}}{2\Gamma(1-\alpha)}\Big(\|e^{-(\kappa(\x)+2C_{\kappa})T}\eta_{u}(T,\x)\|^{2}+\|e^{\kappa(\x)\tau}\eta_{u}(\tau,\x)\|^{2}\Big)d\tau\\	
%			&~~~~+\int_{0}^{T}\int_{0}^{t}\frac{(t-\tau)^{-\alpha}}{2\Gamma(1-\alpha)}\Big(\|e^{\kappa(\x)\tau}\eta_{u}(\tau,\x) \|^{2}+\Big\|\partial_{t}\Big(e^{-(\kappa(\x)+2C_{\kappa})t} \eta_{u}(t,\x)\Big)\Big\|^{2} \Big)d\tau dt.\\
		\end{split}	
	\end{equation*}
By using the Cauchy-Schwarz inequality, the boundedness of $\kappa(\x)$, and Lemma\,\ref{projection_err_u}, one can obtain
	$\uppercase\expandafter{\romannumeral6}\leq Ch^{2k+2}$. Hence
	$\uppercase\expandafter{\romannumeral3}\leq Ch^{2k+2}$.
	Moreover,  taking function $u(t,\x)=e^{-C_{\kappa}t}\zeta_{u}(t,\x)$ in Lemma\,\ref{L2-1}, there exists 
	\begin{equation*}\label{semi-discrete error15-9}
		\begin{split}
			\uppercase\expandafter{\romannumeral3}&\geq C_{\alpha}T^{-\alpha}\int_{0}^{T}\|e^{-C_{\kappa}t}\zeta_{u}(t,\x)\|^{2}dt\\
			&\geq C_{\alpha}C_{\min}^{2}T^{-\alpha}\int_{0}^{T}\|\zeta_{u}(t,\x)\|^{2}dt.\\
		\end{split}	
	\end{equation*}
Thus
	\begin{equation*}\label{semi-discrete error15-10}
		\int_{0}^{T}\|\zeta_{u}(t,\x)\|^{2}dt\leq Ch^{2k+2}.
	\end{equation*}
Then the proof can be completed after using the triangle inequality.
	
\end{proof}

%=========================================================================================

\section{The fully discrete LDG scheme}
\label{Sec4}
%\subsection{Subsection title}
%\label{sec:4_1}

In this section, we first introduce the fully discrete scheme for Eq.~\eqref{equivalent}. Then, the stability analysis of the fully discrete scheme is proposed. Based on the obtained error estimates of semi-discrete scheme, we provide the error estimates of the fully discrete scheme.

\subsection{ Construction of the fully discrete scheme}
\label{Sec4_1}

Let $M$ be a positive integer. Assume that  $t_{n}=(n/M)^{\gamma}T$ for $n = 0,1,\cdots, M$, is the mesh point of the interval $[0,T]$, where the constant mesh grading $\gamma>1$. In particular, when $\gamma=1$, graded mesh is uniform one. Set step size $\tau_{n}=t_{n}-t_{n-1}$ and the maximum step size $\tau:=\max\limits_{1\leq i \leq M}\tau_{i}$. An approximation to the time fractional derivative $ {}_0^C D_t^{\alpha,\kappa(\x)}u(t,\x) $ in Eq.~\eqref{equivalent}, called the $ L1 $ scheme, can be obtained by a simple quadrature formula as

\begin{equation}\label{fully-discrete1}
	\begin{split}
		{}_0^C D_t^{\alpha,\kappa(\x)}u(t_{n},\x)&=\frac{e^{-\kappa(\x)t_n}}{\Gamma(1-\alpha)}\sum_{i=1}^{n}\int_{t_{i-1}}^{t_i}\frac{\partial}{\partial_{s}}[e^{\kappa(\x)s}u(s,\x)]\cdot\frac{ds}{(t_n-s)^\alpha}\\
		&=\frac{e^{-\kappa(\x)t_n}}{\Gamma(1-\alpha)}\sum_{i=1}^{n}\frac{e^{\kappa(\x)t_{i}}u(t_{i},\x)-e^{\kappa(\x)t_{i-1}}u(t_{i-1},\x)}{\tau}\\
		&~~~~\cdot\int_{t_{i-1}}^{t_i}\frac{ds}{(t_n-s)^\alpha}+\Upsilon^{n}(\x)\\
		&=\frac{e^{-\kappa(\x)t_n}}{\Gamma(2-\alpha)}\sum_{i=1}^{n}\frac{e^{\kappa(\x)t_{i}}u(t_{i},\x)-e^{\kappa(\x)t_{i-1}}u(t_{i-1},\x)}{\tau}\\
		&~~~~\cdot[(t_n-t_{i-1})^{1-\alpha}-(t_n-t_{i})^{1-\alpha}]+\Upsilon^{n}(\x),\\
		&=:{}^CD_M^{\alpha,\kappa(\x)}u(t_{n},\x)+\Upsilon^{n}(\x)\\
	\end{split}
\end{equation}
where $\Upsilon^{n}(\x)$ is  temporal truncation error of the $ L1 $ scheme of $ {}_0^C D_t^{\alpha,\kappa(\x)}u(t,\x) $  at $t=t_n$. By simple calculations, one can get
\begin{equation}\label{fully-discrete1-3}
	\begin{split}
		{}^C D_M^{\alpha,\kappa(\x)}u(t_{n},\x)&:=A_{0}^{n}u(t_n,\x)-A_{n-1}^{n}e^{-\kappa(\x)t_n}u(t_{0},\x)\\
		&~~~~+\sum_{i=1}^{n-1}(A_{i}^{n}-A_{i-1}^{n})e^{-\kappa(\x)(t_{n}-t_{n-i})}u(t_{n-i},\x),
	\end{split}
\end{equation}
where
\begin{equation}\label{fully-discrete1-3-1}
	A_{i-1}^{n}=\frac{(t_{n}-t_{n-i})^{1-\alpha}-(t_{n}-t_{n-i+1})^{1-\alpha}}{\Gamma(2-\alpha)\tau_{n-i+1}},\quad i=1,2,\cdots,n.
\end{equation}

Using the mean value theorem, one can easily prove that
\begin{equation}\label{fully-discrete1-4}
	A_{0}^{n}\geq A_{1}^{n}\geq\cdots\geq A_{n-1}^{n}>0,\quad 1\leq n\leq M.
\end{equation}
Based on the spatial semi-discrete LDG scheme \eqref{semi-discrete} and \eqref{fully-discrete1-3}, one can obtain
\begin{equation}\label{fully-discrete2}
	\left\{
	\begin{aligned}
		& ({}^C D_M^{\alpha,\kappa(\x)}u_h(t_{n},\x),v)_{\Omega_{ij}}=-(\Upsilon^{n}(\x),v)_{\Omega_{ij}}-(\textit{\textbf{p}}_h(t_n,\x),\nabla v)_{\Omega_{ij}}\\
		&~~~~~~~~~~~~~~~~~~~~~~~~~~~~~~~~~+\langle\widehat{\textit{\textbf{p}}_h}(t_{n},\x)\cdot\textit{\textbf{n}}, v\rangle_{\partial\Omega_{ij}}, \\
		&(\textit{\textbf{p}}_h(t_{n},\x),\textit{\textbf{w}})_{\Omega_{ij}}+(u_h(t_{n},\x),\nabla\cdot\textit{\textbf{w}})_{\Omega_{ij}}-\langle \widehat{u_h}(t_{n},\x),\textit{\textbf{w}}\cdot\textit{\textbf{n}}\rangle_{\partial\Omega_{ij}}=0,\\
		&(u_h(0),v)_{\Omega_{ij}}=(u_{0}(\x),v)_{\Omega_{ij}},
	\end{aligned}\right.
\end{equation}
where $ (v,\textit{\textbf{w}})\in V_{h}^{k}\times \textit{\textbf{V}}_{h}^{k} $ for $ \Omega_{ij}\in\Omega_{h}$.
Denote the approximation of $(u(t_{n},\x),\textit{\textbf{p}}(t_{n},\x))$ by $(u_{h}^{n},\textit{\textbf{p}}_{h}^{n})\in V_{h}^{k}\times \textit{\textbf{V}}_{h}^{k}$. With the semi-discrete scheme \eqref{semi-discrete}, there exists the fully discrete LDG scheme:

%	 find $(u_{h}^{n},\textit{\textbf{p}}_{h}^{n})\in V_{h}^{k}\times \textit{\textbf{V}}_{h}^{k}$,

\begin{equation}\label{fully-discrete3}
	\left\{
	\begin{aligned}
		& ({}^C D_M^{\alpha,\kappa(\x)}u_h^n,v)_{\Omega_{ij}}+(\textit{\textbf{p}}_h^n,\nabla v)_{\Omega_{ij}}-\langle\widehat{\textit{\textbf{p}}_h^n}\cdot\textit{\textbf{n}}, v\rangle_{\partial\Omega_{ij}}=0  , \\
		&(\textit{\textbf{p}}_h^n,\textit{\textbf{w}})_{\Omega_{ij}}+(u_h^n,\nabla\cdot\textit{\textbf{w}})_{\Omega_{ij}}-\langle \widehat{u_h^n},\textit{\textbf{w}}\cdot\textit{\textbf{n}}\rangle_{\partial\Omega_{ij}}=0 ,\\
		&(u_h^0,v)_{\Omega_{ij}}=(u_{0}(\x),v)_{\Omega_{ij}},
	\end{aligned}\right.
\end{equation}
where  $ (v,\textit{\textbf{w}})\in V_{h}^{k}\times \textit{\textbf{V}}_{h}^{k} $ for $ \Omega_{ij}\in\Omega_{h}$, and $ \widehat{u_{h}^{n}} $ and  $ \widehat{\textit{\textbf{p}}_{h}^{n}} $ are defined in \eqref{fluxes_1}.

\subsection{ Stability analysis of the fully discrete scheme}
\label{Sec4_2}

Similar to the previous numerical analysis of the spatial semi-discrete scheme, we will analyze the effectiveness of the fully discrete scheme \eqref{fully-discrete3} in theory. We prove that the fully discrete scheme is stable, which is the key to ensure that the numerical implementation can be effectively performed.

\begin{theorem}[$L^{2}$-stability] For periodic boundary condition, the fully discrete LDG scheme \eqref{fully-discrete3} is unconditionally stable. There exists a positive constant $ C_{max} $ depending on $\kappa(\x)$, such that
	the numerical solution $u_{h}^{n}$ satisfies
	\begin{equation*}\label{fullystability}
		\|u_{h}^{n}\|\leq C_{max} \|u_{h}^{0}\|, \quad n\geq1,
	\end{equation*}
	where $C_{max} $ is defined by \eqref{assumptions2}.
\end{theorem}
\begin{proof}
	Taking test functions $v=u_{h}^{n}$, $\textit{\textbf{w}}= \textit{\textbf{p}}_{h}^{n}$, and summing all terms of \eqref{fully-discrete3}, we have
	\begin{equation}\label{fullystability1}
		-B(u_{h}^{n},u_{h}^{n};\textit{\textbf{p}}_{h}^{n},\textit{\textbf{p}}_{h}^{n})= ({}^C D_M^{\alpha,\kappa(\x)}u_{h}^{n},u_{h}^{n})_{\Omega}+\|\textit{\textbf{p}}_{h}^{n}\|^{2},
	\end{equation}
	where $B(u_{h}^{n},u_{h}^{n};\textit{\textbf{p}}_{h}^{n},\textit{\textbf{p}}_{h}^{n})$ is defined by \eqref{semi-discrete Stability1}.
		By using Lemma\,\ref{bilinear}, one can get
	\begin{equation}\label{fullystability2}
		B(u_{h}^{n},u_{h}^{n};\textit{\textbf{p}}_{h}^{n},\textit{\textbf{p}}_{h}^{n})=0.
	\end{equation}
	After a simple calculation, there exists 
	\begin{equation}\label{fullystability3}
		\begin{split}
			A_{0}^{n}\|u_{h}^{n}\|^2&\leq\sum_{i=1}^{n-1}(A_{i-1}^{n}-A_{i}^{n})(e^{-\kappa(\x)(t_{n}-t_{n-i})}u_{h}^{n-i},u_{h}^{n})+A_{n-1}^{n}(e^{-\kappa(\x)t_n}u_{h}^{0},u_{h}^{n})\\
			&	\leq\sum_{i=1}^{n-1}(A_{i-1}^{n}-A_{i}^{n})\|e^{-\kappa(\x)(t_{n}-t_{n-i})}u_{h}^{n-i}\|\cdot\|u_{h}^{n}\|\\
			&~~~~+A_{n-1}^{n}\|e^{-\kappa(\x)t_n}u_{h}^{0}\|\cdot \|u_{h}^{n}\|,
		\end{split}
	\end{equation}
	where \eqref{fully-discrete1-3}, \eqref{fully-discrete1-4}, \eqref{fullystability1} and \eqref{fullystability2} are used.
	
	Next, we use mathematical induction to prove the stability of the numerical scheme.
	For $n=1$, there is only the second term on the right side of \eqref{fullystability3}. By using \eqref{assumptions1}, one can get
	\begin{equation*}\label{fullystability6}
		\|u_{h}^{1}\|\leq e^{C_\kappa t_1}\|u_{h}^{0}\|.
	\end{equation*}
	Now, we suppose the following inequalities hold
	\begin{equation}\label{fullystability8}
		\|u_{h}^{m}\|\leq e^{C_{\kappa}t_{m}}\| u_{h}^{0}\|,\quad m=1,\cdots,n-1.
	\end{equation}
		For $m=n$, plugging \eqref{fullystability8} into \eqref{fullystability3} leads to
	\begin{equation*}\label{fullystability10}
		\begin{split}
			A_{0}^{n}\|u_{h}^{n}\|
			&\leq \sum_{i=1}^{n-1}(A_{i-1}^{n}-A_{i}^{n})e^{C_{\kappa}(t_{n}-t_{n-i})}\|u_{h}^{n-i}\|+A_{n-1}^{n}e^{C_{\kappa}t_{n}}\|u_{h}^{0}\|\\
			&\leq\sum_{i=1}^{n-1}(A_{i-1}^{n}-A_{i}^{n})e^{C_{\kappa}(t_{n}-t_{n-i})}e^{C_{\kappa}t_{n-i}}\|u_{h}^{0}\|+A_{n-1}^{n}e^{C_{\kappa}t_{n}}\|u_{h}^{0}\|\\
			&=A_{0}^{n}e^{C_{\kappa}t_{n}}\| u_{h}^{0}\|.\\
		\end{split}
	\end{equation*}	
	Thus,
	\begin{equation*}\label{fullystability11}
		\|u_{h}^{n}\|
		\leq e^{C_{\kappa}t_{n}}\| u_{h}^{0}\|\leq C_{max}\| u_{h}^{0}\|,
	\end{equation*}
	where $C_{max}$ is defined by \eqref{assumptions2}.
		This finishes the proof of the stability result.
\end{proof}

\subsection{Error estimates of the fully discrete scheme}
\label{Sec4_3}

 Now, we consider the numerical convergence of the fully discrete scheme \eqref{fully-discrete3}. The coefficient \eqref{fully-discrete1-3-1} is used to define a convolutional coefficient recursively \cite{Ren-2021-389}. For $n\geq1$, 
\begin{equation}\label{fullyconvergence0}
	P_{0}^{n}=\dfrac{1}{A_{0}^{n}}, \quad 	P_{n-j}^{n}=\dfrac{1}{A_{0}^{j}}\sum_{i=j+1}^{n}(A_{i-j-1}^{i}-A_{i-j}^{i})P_{n-i}^{n},\quad 1\leq j \leq n-1.	
\end{equation}
The discrete coefficient $P_{n-j}^{n}$ is defined to simulate the convolution kernel of the Riemann-Liouville fractional integral of order $ \alpha $. According to \eqref{fullyconvergence0}, it is not difficult to find that the discrete convolution kernel $P_{n-j}^{n}$ satisfies the following properties
\begin{equation}\label{fullyconvergence0_1}
	\sum_{j=m}^{n}P_{n-j}^{n}A_{j-m}^{j}\equiv1, \quad 1\leq m\leq n\leq M.
\end{equation}

In fact, one can also derive \eqref{fullyconvergence0} based on \eqref{fullyconvergence0_1}  and both of them can be used as the definition of discrete convolution kernel.

\begin{lemma} \label{kernel_1}
	Let $P_{n-j}^{n}$ and $ A_{i-1}^{n}$ be defined by \eqref{fullyconvergence0} and \eqref{fully-discrete1-3-1}. Then
	\begin{equation}\label{fullyconvergence1}
		P_{0}^{n}\cdot\sum_{i=1}^{n-j}(A_{i-1}^{n}-A_{i}^{n})P_{n-i-j}^{n-i}=P_{n-j}^{n},\quad i,j=1,2,\cdots,n-1.
	\end{equation}
\end{lemma}

\begin{proof}
	It is straightforward to evaluate. Denoting the left-hand side of \eqref{fullyconvergence1} by $L$, for $1\leq m \leq n-1$, one can obtain
	\begin{equation*}\label{fullyconvergence1_1}
		\begin{split}
			\sum_{j=m}^{n-1}L\cdot A_{j-m}^{j}	&=\sum_{j=m}^{n-1}\Big[P_{0}^{n}\cdot\sum_{i=1}^{n-j}(A_{i-1}^{n}-A_{i}^{n})P_{n-i-j}^{n-i}\Big]\cdot A_{j-m}^{j}\\
			&=\sum_{i=1}^{n-m}P_{0}^{n}\cdot\Big[\sum_{j=m}^{n-i}P_{n-i-j}^{n-i}A_{j-m}^{j}\Big]\cdot(A_{i-1}^{n}-A_{i}^{n}) \\
			&=\sum_{i=1}^{n-m}P_{0}^{n}\cdot(A_{i-1}^{n}-A_{i}^{n})\\
			&=1-P_{0}^{n}\cdot A_{n-m}^{n}.
		\end{split}
	\end{equation*}
		On the other hand, according to \eqref{fullyconvergence0_1}, there exists 
	\begin{equation*}\label{fullyconvergence1_2}
		\sum_{j=m}^{n-1}P_{n-j}^{n}\cdot A_{j-m}^{j}=1-P_{0}^{n}\cdot A_{n-m}^{n},
	\end{equation*}
	which concludes the proof.
\end{proof}

%	引理\ref{truncation error} 的结论表明, 全局截断误差的收敛阶与方程解的正则性密切相关. 对于正则性较好的解来说, 使用较小的梯度网格参数就能够保证误差的收敛阶达到最优. 这一点在数值实验的表格\eqref{example_2_7} 中我们可以清晰地看到.		
Next Lemma suggests that the convergence order of global truncation error of L1 scheme is closely related to the regularity of the solution; even using small mesh parameter $\gamma$ the optimal time accuracy of order $O(\tau^{2-\alpha})$  can be achieved under the condition that the regularity of the solution is well. This result can be clearly observed in Table \ref{example_2_7} of Section \ref{Sec5}.	

\begin{lemma}\label{truncation error}
	For $u\in C_{\delta}^2((0,T];H^{2}(\Omega_{h}))$, denoting  $\Upsilon^{n}(\x)$ as the temporal truncation error of the $ L1 $ scheme of $ {}_0^C D_t^{\alpha,\kappa(\x)}u(t,\x) $ at $t=t_n$, there exists a constant $C>0$, such that
	\begin{equation}\label{fullyconvergence1_3}
		\sum_{j=1}^{n}P_{n-j}^{n}\|\Upsilon^{j}(\x)\|\leq\dfrac{C}{\delta(1-\alpha)}\tau^{\min\{2-\alpha,\gamma\delta\}}, \quad 1\leq n\leq M,
	\end{equation}
	where $ C_{\delta}^2((0,T])= \{u(t)\big|\ u\in  C^2((0,T]),|u^{(l)}|\leq C_u(1+t^{\delta-l}),\ l=1,2,\ 0<t\leq T\}$ and the regularity parameter $ \delta\in(0,1)\cup(1,2) $.
\end{lemma}
\begin{proof}
	The estimate \eqref{fullyconvergence1_3} can be obtained by combining \eqref{assumptions2} and the proof process of \cite[Lemma 3.2]{Ren-2021-389}. 	
\end{proof}	

We will use the convergence results of semi-discrete scheme obtained in the previous section to prove the convergence of fully discrete scheme \eqref{fully-discrete3}. Next, we only discuss the relationship between the solution of the spatial semi-discrete scheme and the solution of the fully discrete scheme, that is, the error estimates of temporal discretization.		

\begin{theorem} \label{temporal discretization}
	Let $u_h\in C_{\delta}^2((0,T];V_{h}^{k}) $ and $u_{h}^{n}\in V_{h}^{k}$ be the numerical solutions of the spatial semi-discrete scheme \eqref{semi-discrete} and the fully discrete scheme \eqref{fully-discrete3}, respectively. Then the following estimate holds	
	
	\begin{equation*}\label{fullyconvergence}
		\|u_h(t_n,\x)-u_{h}^{n}(\x)\|\leq C\tau^{\min\{2-\alpha,\gamma\delta\}},\quad n=1,\cdots,M,
	\end{equation*}
	where $C$ is a constant independent of $\tau$.
\end{theorem}

\begin{proof}
	Suppose that $(u_h,\textbf{\textit{p}}_h)$ and $(u_h^n,\textbf{\textit{p}}_h^n)$ are the numerical solutions of the scheme \eqref{semi-discrete}
	and \eqref{fully-discrete3}, respectively. Then $(u_h(t_n,\x),\textbf{\textit{p}}_h(t_n,\x))$ satisfies Eq.~\eqref{fully-discrete2}.
	Combining
	\eqref{fully-discrete3} with \eqref{fully-discrete2} implies that
	\begin{equation}\label{fullyconvergence3}
		-B(e_{u}^{n},v;e_{\textit{\textbf{p}}}^{n},\textit{\textbf{w}})= ({}^C D_M^{\alpha,\kappa(x)}e_{u}^{n},v)+(e_{\textit{\textbf{p}}}^{n},\textit{\textbf{w}})+(\Upsilon^{n}(\x),v),
	\end{equation}
	where $e_{u}^{n}=u_h(t_{n},\x)-u_{h}^{n}$, $e_{\textit{\textbf{p}}}^{n}=\textit{\textbf{p}}_h(t_n,\x)-\textit{\textbf{p}}_h^{n}$.
	
	Taking the test functions $v=e_{u}^{n}\in V_{h}^{k} $, $\textit{\textbf{w}}=e_{\textit{\textbf{p}}}^{n}\in \textit{\textbf{V}}_{h}^{k} $ in \eqref{fullyconvergence3} leads to
	\begin{equation*}\label{fullyconvergence3-2}
		\begin{split}
			({}^C D_M^{\alpha,\kappa(\x)}e_{u}^{n},e_{u}^{n})_{\Omega}+\|e_{\textit{\textbf{p}}}^{n}\|^{2}&=-(\Upsilon^{n}(\x),e_{u}^{n}),
			\\
		\end{split}
	\end{equation*}
	where we use the fact that $B(e_{u}^{n},e_{u}^{n};e_{\textbf{\textit{p}}}^{n},e_{\textbf{\textit{p}}}^{n})=0$.
	By using \eqref{fully-discrete1-3} and the fact that $e_{u}^{0}=0$, one can obtain
	
	\begin{equation*}\label{fullyconvergence4}
		\begin{split}
			A_{0}^{n}\|e_{u}^{n}\|^{2}&\leq\sum_{i=1}^{n-1}(A_{i-1}^{n}-A_{i}^{n})\Big(e^{-\kappa(\x)(t_{n}-t_{n-i})}e_{u}^{n-i},e_{u}^{n}\Big)+\|\Upsilon^{n}(\x)\|\cdot\|e_{u}^{n}\|.\\
		\end{split}
	\end{equation*}
	Further, there exists
	\begin{equation*}\label{fullyconvergence5}
		\begin{split}
			A_{0}^{n}\|e_{u}^{n}\|&\leq\sum_{i=1}^{n-1}(A_{i-1}^{n}-A_{i}^{n})e^{C_{\kappa}(t_{n}-t_{n-i})}\|e_{u}^{n-i}\|+\|\Upsilon^{n}(\x)\|,\\
		\end{split}
	\end{equation*}
	where $ C_{\kappa} $ is defined by \eqref{assumptions1}.

	Let $ \theta^{n-i} =e^{C_{\kappa}(t_{n}-t_{n-i})}\|e_{u}^{n-i}\|$. Then $\theta^{n}=\|e_{u}^{n}\|$ and one can get
	\begin{equation}\label{fullyconvergence6}
		\begin{split}
			A_{0}^{n}\theta^{n}&\leq\sum_{i=1}^{n-1}(A_{i-1}^{n}-A_{i}^{n})\theta^{n-i}+\|\Upsilon^{n}(\x)\|.\\
		\end{split}
	\end{equation}	
	Next, we prove $ \theta^{n}\leq \sum_{j=1}^{n}P_{n-j}^{n}\cdot\|\Upsilon^{j}(\x)\|$ by mathematical induction.	For $n=1$, there is only the second term on the right side of \eqref{fullyconvergence6}. By using \eqref{fullyconvergence0}, one can get $\theta^{1}\leq P_{0}^{1}\cdot\|\Upsilon^{1}(\x)\|$.
	Now, we suppose %the following inequalities hold
	\begin{equation}\label{fullystability7}
		\theta^{m}\leq \sum_{j=1}^{m}P_{m-j}^{m}\cdot\|\Upsilon^{j}(\x)\|,\quad m=1,\cdots,n-1.
	\end{equation}
	For $m=n$, plugging \eqref{fullystability7}  into \eqref{fullyconvergence6} leads to
	\begin{equation*}\label{fullyconvergence8}
		\begin{split}
			A_{0}^{n}\theta^{n}&\leq\sum_{i=1}^{n-1}(A_{i-1}^{n}-A_{i}^{n})\sum_{j=1}^{n-i}P_{n-i-j}^{n-i}\cdot\|\Upsilon^{j}(\x)\|+\|\Upsilon^{n}(\x)\|\\
			&=\sum_{j=1}^{n-1}\sum_{i=1}^{n-j}(A_{i-1}^{n}-A_{i}^{n})\cdot P_{n-i-j}^{n-i}\cdot\|\Upsilon^{j}(\x)\|+\|\Upsilon^{n}(\x)\|\\
		\end{split}
	\end{equation*}	
According to \eqref{fullyconvergence0} and Lemma\,\ref{kernel_1}, there exists
	\begin{equation*}\label{fullystability9}
		\theta^{n}\leq \sum_{j=1}^{n}P_{n-j}^{n}\cdot\|\Upsilon^{j}(\x)\|,\quad n=1,\cdots,M.
	\end{equation*}	
	
	Finally, by using Lemma\,\ref{truncation error}, it holds	
	\begin{equation*}\label{fullystability9_1}
		\|e_{u}^{n}\|\leq C\tau^{\min\{2-\alpha,\gamma\delta\}},\quad n=1,\cdots,M,
	\end{equation*}	
	which completes the proof.
\end{proof}

\begin{theorem}\label{error} Let $u(t,\x)\in C_{\delta}^2((0,T];H^{s+1}(\Omega_{h})\cap H^{2}(\Omega_{h})) $ and $u_{h}^{n}$  be the exact solution of Eq.~\eqref{equivalent} and the numerical solution of the fully discrete scheme \eqref{fully-discrete3}, respectively. Then for\,$ s\geq k $, the following estimate holds
	\begin{equation}\label{fullyconvergence9}
		\sum_{n=1}^{M}\int_{t_n-1}^{t_n}\|\Pi_{1,n}u(t,\x)-\bar{u}_{h}(t,\x)\|dt\leq C(h^{k+1}+\tau^{\min\{2-\alpha,\gamma\delta\}}),
	\end{equation}
	where $C>0$ is a constant independent of $\tau$, $h$, and $ \Pi_{1,n}u(t,\x) $ and $\bar{u}_{h}(t,\x)$ are linear interpolation functions on the time interval $ (t_{n-1},t_n) $ with  $ \Pi_{1,n}u(t_i,\x)=u(t_i,\x) $,\,$\bar{u}_{h}(t_i,\x)=u_{h}^{i}, i=n-1,n$.
\end{theorem}
\begin{proof}
	Let $u_h$ be the numerical solution of the spatial semi-discrete scheme \eqref{semi-discrete} and $ \Pi_{1,n}u_h(t,\x) $ be linear interpolation function of $u_h(t,\x) $ on the time interval $ (t_{n-1},t_n)$. By using Theorem\,\ref{Th3_2} and the Cauchy-Schwarz inequality, one can get
	\begin{equation*}\label{fullyconvergence10}
		\begin{split}
			\sum_{n=1}^{M}\int_{t_n-1}^{t_n}\|\Pi_{1,n}(u(t,\x)-u_h(t,\x)) \|dt&\leq\int_{0}^{T}\|u(t,\x)-u_{h}(t,\x)\|dt+C_3\tau^2\\
			&\leq C(h^{k+1}+\tau^2).\\	
		\end{split}
	\end{equation*}
According to Theorem\,\ref{temporal discretization}, there exists
	\begin{equation*}\label{fullyconvergence11}
		\begin{split}
			\sum_{n=1}^{M}\int_{t_n-1}^{t_n}\|\Pi_{1,n}u_h(t,\x)-\bar{u}_{h}(t,\x)\|dt&\leq T\cdot\max\limits_{0\leq n\leq M}\|u_h(t_n,\x)-u_{h}^{n}(\x)\|\\
			&\leq C\tau^{\min\{2-\alpha,\gamma\delta\}}.
		\end{split}
	\end{equation*}
Finally, the proof is completed by using the triangle inequality.
\end{proof}

\section{Numerical experiments}
\label{Sec5}
%\subsection{Subsection title}
%\label{sec:5_1}
To justify the theoretical results, we perform extensive numerical experiments. We define the $L^2$ error between the numerical solution $u_h^n$  of the fully discrete scheme \eqref{fully-discrete3} and the exact solution $u(t_n,\x)$ of Eq.~\eqref{equivalent}
\begin{equation*}\label{example_1_0}
	E(h,M)=\sum_{n=1}^{M}\int_{t_{n-1}}^{t_n}\|\Pi_{1,n}u(t,\x)-\bar{u}_{h}(t,\x)\|dt,
\end{equation*}
where $\Pi_{1,n}u$ and $\bar{u}_{h}$ are defined by Theorem \ref{error}.

In the following examples, the $L^2$ errors and the spatial convergence orders of piecewise $Q^0$, $Q^1$, $Q^2$ polynomials to approximate the exact solution are calculated, respectively. In order to verify the effectiveness of the generalized numerical fluxes,  the fully discrete LDG scheme \eqref{fully-discrete3} with the generalized numerical fluxes with different parameters are used.

In particular, the generalized numerical flux is called the central numerical fluxes when $\sigma_1=1/2$, $\sigma_2=1/2$. According to the previous theoretical analyses, for the central numerical fluxes the optimal spatial convergence orders  cannot be obtained. However, we still numerically calculate the $L^2$ errors and the spatial convergence rates of the central numerical fluxes in this section and the numerical results show the effect of the fluxes on the spatial convergence rates. In addition, the influence of the change of the generalized numerical fluxes on the condition number of coefficient matrix is also observed.

\begin{example}\label{Example_1_1_}
	As the first example, we consider the homogeneous Feynman-Kac backward equation
	\begin{equation}\label{example_1_1_}
		{}_0^C D_t^{\alpha,-2}u(t,\x)=\Delta u(t,\x),\quad (t,\x)\in (0,T]\times\Omega,
	\end{equation}
	where $\Omega = [0,1] \times [0,1]$. The boundary condition is periodic and the initial value condition is
	\begin{equation*}\label{example_1_2_}
		u_0(\x)=\sin(2\pi y)\cos(2\pi x) .	
	\end{equation*}
	The exact solution of Eq.~\eqref{example_1_1_} with the above initial boundary value conditions is
	\begin{equation}\label{example_1_3_}
		u(t,\x)=e^{2t}E_{\alpha}(-8\pi^2t^{\alpha})\cos(2\pi x)\sin(2\pi y),	\end{equation}
	where $ E_{\alpha}(t)=\sum_{l=0}^{\infty}\dfrac{t^l}{\Gamma(l\alpha+1)} $, called Mittag-Leffler function.
\end{example}

In order to weaken the influence of temporal errors, we divide the time interval into sufficiently small parts when verifying the convergence rates of spatial errors. In Table \ref{example_1_1}, take $M=20$, $\gamma=2 $, $T=0.1 $,  and approximate the exact solution using piecewise $Q^0$ polynomial. In Table \ref{example_1_2}, take $M=20$, $\gamma=3 $, $T=1 $, and approximate the exact solution by piecewise $Q^1$ polynomial. In Table \ref{example_1_3}, take $M=100$, $\gamma=6 $, $T=0.1 $, and approximate the exact solution using $Q^2$ polynomial. The numerical simulation results show that the spatial convergence rates in Tables \ref{example_1_1}-\ref{example_1_3} are $ O(h) $, $ O(h^2) $, and $ O(h^3) $, respectively, which are in good agreement with the theoretical results \eqref{fullyconvergence9}.

Comparing Table \ref{example_1_1} with Table \ref{example_1_3}, it is not difficult to find that the change of  parameters of numerical fluxes has little effect on the spatial convergence rates using piecewise $Q^0$ and $Q^2$ polynomials, which seems to be contrary to the results by piecewise $Q^1$ polynomial in Table \ref{example_1_2}.

\begin{table}
	\caption{The $L^2$ errors and spatial convergence orders using piecewise $Q^0$ polynomial with the generalized alternating numerical fluxes in Example \ref{Example_1_1_}.}
	\label{example_1_1}       % Give a unique label
	\begin{tabular}{lllll}
		\hline\noalign{\smallskip}
		$(\sigma_1,\sigma_2,\alpha)$$ \backslash $$h$ & ~~~~~~~~~1/12 & ~~~~~~~~~~~1/14 & ~~~~~~~~~~~1/16 & ~~~~~~~~~~~1/18 \\
		\noalign{\smallskip}\hline\noalign{\smallskip}
		~~$ (0,1,0.7)$ & ~~~~~2.9071e-04 &~~~~~~~2.4957e-04  &~~~~~~~2.1866e-04  &~~~~~~~1.9459e-04   \\
		 &~~~~~~~~Rates   &~~~~~~~~~~0.9898 &~~~~~~~~~~0.9903 &~~~~~~~~~~0.9901 \\
			\noalign{\smallskip}\hline\noalign{\smallskip}
		$ (0.6,0.3,0.7)$ & ~~~~~3.1971e-04 &~~~~~~~2.6880e-04  &~~~~~~~2.3232e-04  &~~~~~~~2.0483e-04   \\
		&~~~~~~~~Rates   &~~~~~~~~~~1.1253 &~~~~~~~~~~1.0921 &~~~~~~~~~~1.0694 \\
			\noalign{\smallskip}\hline\noalign{\smallskip}
		
		$ (0.4,0.5,0.7)$ & ~~~~~3.2375e-04 &~~~~~~~2.7139e-04  &~~~~~~~2.3412e-04  &~~~~~~~2.0614e-04   \\
		&~~~~~~~~Rates   &~~~~~~~~~~1.1444 &~~~~~~~~~~1.1065 &~~~~~~~~~~1.0806 \\
		\noalign{\smallskip}\hline
		
	\end{tabular}
\end{table}

\begin{table}
	\caption{The $L^2$ errors and spatial convergence orders using piecewise $Q^1$ polynomial with the generalized alternating numerical fluxes in Example \ref{Example_1_1_}.}
	\label{example_1_2}       % Give a unique label
	\begin{tabular}{lllll}
		\hline\noalign{\smallskip}
		$(\sigma_1,\sigma_2,\alpha)$$ \backslash $$h$ & ~~~~~~~~~1/12 & ~~~~~~~~~~~1/14 & ~~~~~~~~~~~1/16 & ~~~~~~~~~~~1/18 \\
		\noalign{\smallskip}\hline\noalign{\smallskip}
		~~$ (0,1,0.5)$ & ~~~~~2.2780e-04 &~~~~~~~1.6766e-04  &~~~~~~~1.2852e-04  &~~~~~~~1.0165e-04   \\
		&~~~~~~~~Rates   &~~~~~~~~~~1.9886 &~~~~~~~~~~1.9908 &~~~~~~~~~~1.9917 \\
		\noalign{\smallskip}\hline\noalign{\smallskip}
    	~~~~~------ & ~~~~~~~~~1/20 & ~~~~~~~~~~~1/22 & ~~~~~~~~~~~1/24 & ~~~~~~~~~~~1/26 \\
		\noalign{\smallskip}\hline\noalign{\smallskip}
		$ (0.1,0.8,0.5)$ & ~~~~~1.0563e-04 &~~~~~~~8.7725e-05
		&~~~~~~~7.4005e-05
		&~~~~~~~6.3265e-05   \\
		&~~~~~~~~Rates   &~~~~~~~~~~1.9484 &~~~~~~~~~~1.9547 &~~~~~~~~~~1.9589 \\
		\noalign{\smallskip}\hline\noalign{\smallskip}
		~~~~~------ & ~~~~~~~~~1/32 & ~~~~~~~~~~~1/36 & ~~~~~~~~~~~1/40 & ~~~~~~~~~~~1/44 \\
		\noalign{\smallskip}\hline\noalign{\smallskip}
		$ (0.3,0.7,0.5)$ & ~~~~~6.3955e-05 &~~~~~~~5.1027e-05  &~~~~~~~4.1653e-05  &~~~~~~~3.4652e-05   \\
		&~~~~~~~~Rates   &~~~~~~~~~~1.9173 &~~~~~~~~~~1.9266 &~~~~~~~~~~1.9306 \\
		\noalign{\smallskip}\hline
	\end{tabular}
\end{table}

\begin{table}
\caption{The $L^2$ errors and spatial convergence orders using piecewise $Q^2$ polynomial with the generalized alternating numerical fluxes in Example \ref{Example_1_1_}.}
	\label{example_1_3}       % Give a unique label
	\begin{tabular}{lllll}
		\hline\noalign{\smallskip}
		$(\sigma_1,\sigma_2,\alpha)$$ \backslash $$h$ & ~~~~~~~~~1/10 & ~~~~~~~~~~~1/12 & ~~~~~~~~~~~1/14 & ~~~~~~~~~~~1/16 \\
		\noalign{\smallskip}\hline\noalign{\smallskip}
		~~$ (1,0,0.3)$ & ~~~~~1.9815e-06 &~~~~~~~1.1514e-06  &~~~~~~~7.2730e-07  &~~~~~~~4.8878e-07   \\
		&~~~~~~~~Rates   &~~~~~~~~~~2.9778 &~~~~~~~~~~2.9800 &~~~~~~~~~~2.9762 \\
		\noalign{\smallskip}\hline\noalign{\smallskip}
		$ (0.8,0.3,0.3)$ & ~~~~~1.5636e-06 &~~~~~~~8.9614e-07  &~~~~~~~5.6167e-07  &~~~~~~~3.7590e-07   \\
		&~~~~~~~~Rates   &~~~~~~~~~~3.0532 &~~~~~~~~~~3.0306 &~~~~~~~~~~3.0076 \\
		\noalign{\smallskip}\hline\noalign{\smallskip}
		
		$ (0.6,0.5,0.3)$ & ~~~~~1.3772e-06 &~~~~~~~7.8569e-07  &~~~~~~~4.9134e-07  &~~~~~~~3.2858e-07   \\
		&~~~~~~~~Rates   &~~~~~~~~~~3.0782 &~~~~~~~~~~3.0452 &~~~~~~~~~~3.0132 \\
		\noalign{\smallskip}\hline
		
	\end{tabular}
\end{table}
		
\begin{example}\label{Example_2_1_}
	Let us further consider the inhomogeneous Feynman-Kac backward equation.
	\begin{equation}\label{example_2_1_}
		{}_0^C D_t^{\alpha,\cos(2\pi x)}u(t,\x)=\Delta u(t,\x)+f(t,\x),\quad (t,\x)\in (0,0.1]\times\Omega,
	\end{equation}
	where $\Omega = [0,1] \times [0,1]$. The boundary condition is periodic and the initial value condition is $u_0(\x)=0$. Let the exact solution of Eq.~\eqref{example_2_1_} be $u(t,\x)=e^{-t\cos(2\pi x)}t^{\delta}\cos(2\pi x)\sin(2\pi y)$. Then one can choose the source term
	\begin{equation*}\label{example_1_5}
		\begin{split}
			f(t,\x)&=e^{-t\cos(2\pi x)}\Big(\dfrac{\Gamma(1+\delta)}{\Gamma(\delta+1-\alpha)}t^{\delta-\alpha}\Big)\cos(2\pi x)\sin(2\pi y)\\
			&~~~~+\Big(2-t\cos(2\pi x)\Big)g(t,\x)+ \Big(2+t\cos(2\pi x)\Big)h(t,\x),\\	
		\end{split}
	\end{equation*}
	where
	\begin{equation*}\label{example_1_6}
		\begin{split}
			&g(t,\x)=4t^{\delta}\pi^2e^{-t\cos(2\pi x)}\cos(2\pi x)\sin(2\pi y),	\\
			&h(t,\x)=4t^{\delta+1}\pi^2e^{-t\cos(2\pi x)}\sin^2(2\pi x)\sin(2\pi y),	\\
		\end{split}
	\end{equation*}	
	and the regularity parameter $ \delta\in(0,1)\cup(1,2) $.
\end{example}

First, we take $\delta=\alpha$ to verify the spatial convergence orders. In Table \ref{example_2_1}, take $M=20$, $\gamma=2 $, and approximate the exact solution using piecewise constant function. In Table \ref{example_2_3}, take $M=20$, $\gamma=3 $, and approximate the exact solution by piecewise $Q^2$ polynomial. In Table \ref{example_2_5}, take $M=100$, $\gamma=6 $, and approximate the exact solution using piecewise $Q^2$ polynomial. The numerical simulation results show that the spatial convergence orders in Tables \ref{example_2_1}-\ref{example_2_5} are $ O(h) $, $ O(h^2) $, and $ O(h^3) $, respectively, which are in good agreement with the theoretical results \eqref{fullyconvergence9}.

In Example \ref{Example_2_1_}, we verify the influence of central fluxes on the  orders of convergence in detail. By choosing different parameters of numerical fluxes in Table \ref{example_2_3}, it can be noted that although the spatial orders of convergence can reach order $ O(h^2) $ stably when using piecewise $Q^1$ polynomial, the closer the numerical fluxes are to the central ones the finer the subdivision required for the spatial convergence orders to reach stability. More than that, the numerical results in Table \ref{example_2_3} also show that the spatial convergence orders are only the first order when the numerical fluxes are taken as the central ones, which is consistent with the existing conclusion \cite{Hesthaven-2008}.

\begin{table}
	\caption{The $L^2$ errors and spatial convergence orders using piecewise $Q^0$ polynomial with the generalized alternating numerical fluxes in Example \ref{Example_2_1_}.}
	\label{example_2_1}       % Give a unique label
	\begin{tabular}{lllll}
		\hline\noalign{\smallskip}
		$(\sigma_1,\sigma_2,\alpha)$$ \backslash $$h$ & ~~~~~~~~~1/12 & ~~~~~~~~~~~1/14 & ~~~~~~~~~~~1/16 & ~~~~~~~~~~~1/18 \\
		\noalign{\smallskip}\hline\noalign{\smallskip}
		~~$ (1,0,0.3)$ & ~~~~~5.3743e-03 &~~~~~~~4.6124e-03  &~~~~~~~4.0391e-03  &~~~~~~~3.5923e-03   \\
		&~~~~~~~~Rates   &~~~~~~~~~~0.9918 &~~~~~~~~~~0.9938 &~~~~~~~~~~0.9952 \\
		\noalign{\smallskip}\hline\noalign{\smallskip}
		$ (0.7,0.2,0.3)$ & ~~~~~5.6634e-03 &~~~~~~~4.7934e-03  &~~~~~~~4.1600e-03  &~~~~~~~3.6770e-03   \\
		&~~~~~~~~Rates   &~~~~~~~~~~1.0819 &~~~~~~~~~~1.0615 &~~~~~~~~~~1.0478 \\
		\noalign{\smallskip}\hline\noalign{\smallskip}
		$ (0.3,0.6,0.3)$ & ~~~~~5.7668e-03 &~~~~~~~4.8576e-03  &~~~~~~~4.2026e-03  &~~~~~~~3.7067e-03   \\
		&~~~~~~~~Rates   &~~~~~~~~~~1.1130 &~~~~~~~~~~1.0848 &~~~~~~~~~~1.0660 \\	
		\noalign{\smallskip}\hline\noalign{\smallskip}
		~~$ (0,1,0.7)$ & ~~~~~2.1386e-03 &~~~~~~~1.8356e-03  &~~~~~~~1.6075e-03  &~~~~~~~1.4298e-03   \\
		&~~~~~~~~Rates   &~~~~~~~~~~0.9911 &~~~~~~~~~~0.9933 &~~~~~~~~~~0.9948 \\
		\noalign{\smallskip}\hline\noalign{\smallskip}
		$ (0.3,0.8,0.7)$ & ~~~~~2.2450e-03 &~~~~~~~1.9022e-03  &~~~~~~~1.6520e-03  &~~~~~~~1.4609e-03   \\
		&~~~~~~~~Rates   &~~~~~~~~~~1.0748 &~~~~~~~~~~1.0562 &~~~~~~~~~~1.0438 \\
		\noalign{\smallskip}\hline\noalign{\smallskip}
		
		$ (0.4,0.5,0.7)$ & ~~~~~2.3053e-03 &~~~~~~~1.9396e-03  &~~~~~~~1.6768e-03  &~~~~~~~1.4781e-03   \\
		&~~~~~~~~Rates   &~~~~~~~~~~1.1205 &~~~~~~~~~~1.0904 &~~~~~~~~~~1.0704 \\
		\noalign{\smallskip}\hline
		
	\end{tabular}
\end{table}
		
\begin{table}
	\caption{The $L^2$ errors and spatial convergence orders using piecewise $Q^1$ polynomial with the generalized alternating numerical fluxes in Example \ref{Example_2_1_}.}
	\label{example_2_3}       % Give a unique label
	\begin{tabular}{lllll}
		\hline\noalign{\smallskip}
		$(\sigma_1,\sigma_2,\alpha)$$ \backslash $$h$ & ~~~~~~~~~1/12 & ~~~~~~~~~~~1/14 & ~~~~~~~~~~~1/16 & ~~~~~~~~~~~1/18 \\
		\noalign{\smallskip}\hline\noalign{\smallskip}
		~~$ (1,0,0.3)$ & ~~~~~5.9669e-04 &~~~~~~~4.3915e-04  &~~~~~~~3.3660e-04  &~~~~~~~2.6616e-04   \\
		&~~~~~~~~Rates   &~~~~~~~~~~1.9888 &~~~~~~~~~~1.9917 &~~~~~~~~~~1.9935 \\
		\noalign{\smallskip}\hline\noalign{\smallskip}
		~~$ (0,1,0.7)$ & ~~~~~2.3747e-04 &~~~~~~~1.7479e-04  &~~~~~~~1.3399e-04  &~~~~~~~1.0596e-04   \\
		&~~~~~~~~Rates   &~~~~~~~~~~1.9879 &~~~~~~~~~~1.9908 &~~~~~~~~~~1.9925 \\
		\noalign{\smallskip}\hline\noalign{\smallskip}
		$ (0.5,0.5,0.3)$ & ~~~~~1.1287e-03 &~~~~~~~9.5578e-04  &~~~~~~~8.2960e-04  &~~~~~~~7.3330e-04   \\
		&~~~~~~~~Rates   &~~~~~~~~~~1.0786 &~~~~~~~~~~1.0603 &~~~~~~~~~~1.0476 \\	
		\noalign{\smallskip}\hline\noalign{\smallskip}
		$ (0.5,0.5,0.7)$ & ~~~~~4.4600e-04 &~~~~~~~3.7780e-04  &~~~~~~~3.2800e-04  &~~~~~~~2.8997e-04   \\
		&~~~~~~~~Rates   &~~~~~~~~~~1.0765 &~~~~~~~~~~1.0586 &~~~~~~~~~~1.0462 \\
		\noalign{\smallskip}\hline\noalign{\smallskip}
		~~~~~------ & ~~~~~~~~~1/20 & ~~~~~~~~~~~1/22 & ~~~~~~~~~~~1/24 & ~~~~~~~~~~~1/26 \\
		\noalign{\smallskip}\hline\noalign{\smallskip}
		$ (0.9,0.2,0.3)$ & ~~~~~2.7614e-04 &~~~~~~~2.2929e-04  &~~~~~~~1.9337e-04  &~~~~~~~1.6524e-04   \\
		&~~~~~~~~Rates   &~~~~~~~~~~1.9506 &~~~~~~~~~~1.9583 &~~~~~~~~~~1.9643 \\
		\noalign{\smallskip}\hline\noalign{\smallskip}
		$ (0.1,0.8,0.7)$ & ~~~~~1.0993e-04 &~~~~~~~9.1293e-05  &~~~~~~~7.7003e-05  &~~~~~~~6.5812e-05   \\
		&~~~~~~~~Rates   &~~~~~~~~~~1.9491 &~~~~~~~~~~1.9564 &~~~~~~~~~~1.9620 \\
		\noalign{\smallskip}\hline\noalign{\smallskip}
		~~~~~------ & ~~~~~~~~~1/32 & ~~~~~~~~~~~1/36 & ~~~~~~~~~~~1/40 & ~~~~~~~~~~~1/44 \\
		\noalign{\smallskip}\hline\noalign{\smallskip}
		$ (0.7,0.3,0.3)$ & ~~~~~1.6712e-04 &~~~~~~~1.3325e-04  &~~~~~~~1.0865e-04  &~~~~~~~9.0246e-05   \\
		&~~~~~~~~Rates   &~~~~~~~~~~1.9228 &~~~~~~~~~~1.9370 &~~~~~~~~~~1.9475 \\	
		\noalign{\smallskip}\hline\noalign{\smallskip}
		$ (0.2,0.7,0.7)$ & ~~~~~5.7992e-05 &~~~~~~~4.6161e-05  &~~~~~~~3.7604e-05  &~~~~~~~3.1222e-05   \\
		&~~~~~~~~Rates   &~~~~~~~~~~1.9371 &~~~~~~~~~~1.9461 &~~~~~~~~~~1.9514 \\
		\noalign{\smallskip}\hline	
	\end{tabular}
\end{table}

\begin{table}
	\caption{The $L^2$ errors and spatial convergence orders using piecewise $Q^2$ polynomial with the generalized alternating numerical fluxes in Example \ref{Example_2_1_}.}
	\label{example_2_5}       % Give a unique label
	\begin{tabular}{lllll}
		\hline\noalign{\smallskip}
		$(\sigma_1,\sigma_2,\alpha)$$ \backslash $$h$ & ~~~~~~~~~1/10 & ~~~~~~~~~~~1/12 & ~~~~~~~~~~~1/14 & ~~~~~~~~~~~1/16 \\
		\noalign{\smallskip}\hline\noalign{\smallskip}
		~~$ (1,0,0.3)$ & ~~~~~4.4230e-05 &~~~~~~~2.5701e-05  &~~~~~~~1.6225e-05  &~~~~~~~1.0888e-05   \\
		&~~~~~~~~Rates   &~~~~~~~~~~2.9775 &~~~~~~~~~~2.9838 &~~~~~~~~~~2.9876 \\	
		\noalign{\smallskip}\hline\noalign{\smallskip}
		$ (0.8,0.3,0.3)$ & ~~~~~3.5359e-05 &~~~~~~~2.0197e-05  &~~~~~~~1.2621e-05  &~~~~~~~8.4140e-06   \\
		&~~~~~~~~Rates   &~~~~~~~~~~3.0716 &~~~~~~~~~~3.0500 &~~~~~~~~~~3.0366 \\
		\noalign{\smallskip}\hline\noalign{\smallskip}
		$ (0.6,0.5,0.3)$ & ~~~~~3.1299e-05 &~~~~~~~1.7729e-05  &~~~~~~~1.1032e-05  &~~~~~~~7.3361e-06   \\
		&~~~~~~~~Rates   &~~~~~~~~~~3.1175 &~~~~~~~~~~3.0776 &~~~~~~~~~~3.0551 \\
		\noalign{\smallskip}\hline\noalign{\smallskip}
		~~$ (0,1,0.7)$ & ~~~~~1.7602e-05 &~~~~~~~1.0232e-05  &~~~~~~~6.4639e-06  &~~~~~~~4.3426e-06   \\
		&~~~~~~~~Rates   &~~~~~~~~~~2.9754 &~~~~~~~~~~2.9797 &~~~~~~~~~~2.9787 \\
		\noalign{\smallskip}\hline\noalign{\smallskip}
		$ (0.2,0.7,0.7)$ & ~~~~~1.4075e-05 &~~~~~~~8.0437e-06  &~~~~~~~5.0315e-06  &~~~~~~~3.3608e-06   \\
		&~~~~~~~~Rates   &~~~~~~~~~~3.0689 &~~~~~~~~~~3.0436 &~~~~~~~~~~3.0221 \\
		\noalign{\smallskip}\hline\noalign{\smallskip}
		$ (0.4,0.6,0.7)$ & ~~~~~1.2600e-05 &~~~~~~~7.1447e-06  &~~~~~~~4.4523e-06  &~~~~~~~2.9686e-06   \\
		&~~~~~~~~Rates   &~~~~~~~~~~3.1118 &~~~~~~~~~~3.0681 &~~~~~~~~~~3.0355 \\
		\noalign{\smallskip}\hline
		
	\end{tabular}
\end{table}

Theorem \ref{error} shows that the temporal convergence orders  depend on the regularity of the solution. For the solution with well regularity ($\delta$ large) even if the small mesh parameter $\gamma$ is used, we can still achieve an optimal time accuracy of order $  O(\tau^{2-\alpha})  $, that is, when $\gamma\delta\geq2-\alpha $, the temporal convergence orders $  O(\tau^{2-\alpha})  $ can be got. Therefore, we take $\gamma\geq\max{\{(2-\alpha)/\delta,1\}}$ in order to observe the optimal temporal convergence orders. Now, we set the space stepsize as $h=M^{(\alpha-2)/(k+1)}$ in Table \ref{example_2_7} based on the optimal spatial convergence orders. Here, we take parameters of numerical flux $\sigma_1=0$, $\sigma_2=1$, and the degree $ k =0$. The numerical results in Table \ref{example_2_7} show that the temporal convergence rates can reach the optimal accuracy of order $ O(\tau^{2-\alpha}) $, which is consistent with the theoretical result \eqref{fullyconvergence9}.	

\begin{table}
	\caption{The $L^2$ errors and temporal convergence orders using piecewise $Q^0$ polynomial with the generalized alternating numerical fluxes in Example \ref{Example_2_1_}.}
	\label{example_2_7}       % Give a unique label
	\begin{tabular}{lllll}
		\hline\noalign{\smallskip}
		$(\alpha,\delta,\gamma)$$ \backslash $$M$ & ~~~~~~~~~~~~2 & ~~~~~~~~~~~~~~6 & ~~~~~~~~~~~~~9 & ~~~~~~~~~~~~~11 \\
		\noalign{\smallskip}\hline\noalign{\smallskip}
		$ (0.3,0.3,6)$ &~~~~~~2.0005e-02  &~~~~~~~~3.0802e-03  &~~~~~~~1.5421e-03  &~~~~~~~1.0980e-03   \\
		&~~~~~~~~~Rates   &~~~~~~~~~~~1.7031 &~~~~~~~~~~1.7063 &~~~~~~~~~~1.6925 \\	
		\noalign{\smallskip}\hline\noalign{\smallskip}
		$ (0.3,1.4,2)$ & ~~~~~~1.5912e-03 &~~~~~~~~2.4475e-04  &~~~~~~~1.2252e-04  &~~~~~~~8.7230e-05   \\
		&~~~~~~~~~Rates   &~~~~~~~~~~~1.7040 &~~~~~~~~~~1.7067 &~~~~~~~~~~1.6928 \\	
		\noalign{\smallskip}\hline\noalign{\smallskip}
		~~~~------ & ~~~~~~~~~~~~3 & ~~~~~~~~~~~~~~5 & ~~~~~~~~~~~~~9 & ~~~~~~~~~~~~~13 \\
		\noalign{\smallskip}\hline\noalign{\smallskip}
		$(0.5,0.5,3)$ & ~~~~~~7.9394e-03 &~~~~~~~~3.6935e-03  &~~~~~~~1.5126e-03  &~~~~~~~8.6957e-04   \\
		&~~~~~~~~~Rates   &~~~~~~~~~~~1.4981 &~~~~~~~~~~1.5189 &~~~~~~~~~~1.5054 \\
		\noalign{\smallskip}\hline\noalign{\smallskip}
		$(0.5,1.6,1)$ & ~~~~~~6.3123e-04 &~~~~~~~~2.9358e-04  &~~~~~~~1.2019e-04  &~~~~~~~6.9085e-05   \\
		&~~~~~~~~~Rates   &~~~~~~~~~~~1.4986 &~~~~~~~~~~1.5194 &~~~~~~~~~~1.5058 \\
		\noalign{\smallskip}\hline\noalign{\smallskip}
		~~~~------ & ~~~~~~~~~~~~2 & ~~~~~~~~~~~~~~4 & ~~~~~~~~~~~~~8 & ~~~~~~~~~~~~~16 \\
		\noalign{\smallskip}\hline\noalign{\smallskip}
		$ (0.7,0.7,2)$ & ~~~~~~1.0050e-02 &~~~~~~~~4.2074e-03  &~~~~~~~1.7136e-03  &~~~~~~~6.9677e-04   \\
		&~~~~~~~~~Rates   &~~~~~~~~~~~1.2562 &~~~~~~~~~~1.2959 &~~~~~~~~~~1.2983 \\
		\noalign{\smallskip}\hline\noalign{\smallskip}
		$ (0.7,1.8,1)$ & ~~~~~~7.9832e-04 &~~~~~~~~3.3481e-04  &~~~~~~~1.3626e-04  &~~~~~~~5.5376e-05   \\
		&~~~~~~~~~Rates   &~~~~~~~~~~~1.2536 &~~~~~~~~~~1.2970 &~~~~~~~~~~1.2991 \\
		\noalign{\smallskip}\hline	
	\end{tabular}
\end{table}	

In the last part of this section, we reveal the relationship between the condition number of the coefficient matrix of the fully discrete scheme \eqref{fully-discrete3} and the numerical flux and degree of polynomials.

Taking $\{\phi_{i}^r(x)\phi_{j}^s(y)\}_{r,s=0}^k$ as the basis of $Q^k$ on the element $\Omega_{ij}$, our implementation uses the orthogonal Legendre polynomials, and express $	u_h^n(\x)\in V_{h}^{k}$
as
\begin{equation}\label{example_3_1}
	\begin{split}
		u_h^n\big|_{\Omega_{ij}}=\sum_{r=0}^{k}\sum_{s=0}^{k}u_{ij}^{rs}(t_n)\phi_{i}^r(x)\phi_{j}^s(y),
	\end{split}
\end{equation}
with $u_{ij}^{rs}(t_n)$  being the unknown coefficients of numerical solution $u_h^n$.
Introducing $ u_h^n $ \eqref{example_3_1} into the fully discrete scheme \eqref{fully-discrete3} leads to its matrix form
%of the fully discrete problem
\begin{equation}\label{example_3_2}
	\begin{split}
		F_nU^n=\sum_{l=0}^{n-1}G_lU^l,\quad n=1,2,\cdots,M,
	\end{split}
\end{equation}
where $ U^n $ is the
coefficient vector of $u_h^n$.

We take $ M = 100 $, $h=1/12$, $T=0.1$, and calculate the condition number of $F_M$. In Table \ref{example_3_1_}, we find that the condition number of $F_M$ decreases as the numerical fluxes approach the central ones. On the other hand, the condition number of $F_M$ increases gradually with the increase of polynomial degree. 
%In fact, the ill-conditioned  degree of a matrix is determined by the condition number of the matrix, the smaller condition number the lower the ill conditioned degree of a matrix. 
The condition numbers in Table \ref{example_3_1_} are not very large, implying that the equation still can be effectively solved 
% which shows that it is still very effective to calculate the numerical solution of the fully discrete scheme 
 using the matrix form \eqref{example_3_2}.

\begin{table}
	\caption{Condition numbers of the coefficient matrix $F_M$ for the fully scheme \eqref{fully-discrete3}.}
	\label{example_3_1_}       % Give a unique label
	\begin{tabular}{lllll}
		\hline\noalign{\smallskip}
		$(Q^k,\alpha)$$ \backslash $$(\sigma_1,\sigma_2)$ & ~~~~~~~~(1,0) & ~~~~~~~(0.9,0.1) & ~~~~~~(0.8,0.2) & ~~~~~~(0.7,0.3) \\
		\noalign{\smallskip}\hline\noalign{\smallskip}
		~~~~~$ (Q^0,0.3)$ & ~~~2.2489e+02 &~~~~~1.4429e+02  &~~~~8.8317e+01  &~~~~6.5928e+01   \\
		\noalign{\smallskip}\hline\noalign{\smallskip}
		~~~~~$ (Q^1,0.3)$ & ~~~9.3401e+02 &~~~~~8.4303e+02  &~~~~7.7026e+02  &~~~~7.1675e+02   \\
		\noalign{\smallskip}\hline\noalign{\smallskip}
		~~~~~$ (Q^2,0.3)$ & ~~~2.5750e+03 &~~~~~1.8470e+03  &~~~~1.2961e+03  &~~~~9.5453e+02   \\\noalign{\smallskip}\hline\noalign{\smallskip}
		~~~~~$(Q^0,0.7)$ & ~~~1.4294e+01 &~~~~~9.5084e+00  &~~~~6.1848e+00  &~~~~4.8554e+00   \\
		\noalign{\smallskip}\hline\noalign{\smallskip}
		~~~~~$ (Q^1,0.7)$ & ~~~5.6369e+01 &~~~~~5.0972e+01  &~~~~4.6659e+01  &~~~~4.3490e+01   \\
		\noalign{\smallskip}\hline\noalign{\smallskip}
		~~~~~$ (Q^2,0.7)$ & ~~~1.5365e+02 &~~~~~1.1046e+02  &~~~~7.7802e+01  &~~~~5.7456e+01   \\
		\noalign{\smallskip}\hline
		
	\end{tabular}
\end{table}

\section{Conclusion}
\label{Sec6}
%\subsection{Subsection title}
%\label{sec:6_1}

The non-Brownian functional is an important class of statistical observables, the distribution of which is governed by Feynman-Kac equation \cite{Deng-2022}. In this paper, we first derive the equivalent form of the backward Feynman-Kac equation. Then the spatial semi-discrete scheme is constructed by LDG method. Following the properties of the fractional substantial calculus, the stability  and optimal spatial convergence orders $O(h^{k+1})$ of the semi-discrete scheme with the generalized alternating numerical fluxes are obtained. Based on the theoretical results of the semi-discrete scheme, we obtain the optimal convergence orders $O(h^{k+1}+\tau^{\min\{2-\alpha,\gamma\delta\}})$ of the fully discrete scheme. Finally, extensive numerical experiments are carried out to demonstrate the validity of the numerical scheme and justify the theoretical findings. The closer the numerical fluxes are to the central ones the finer the subdivision are required for the spatial convergence orders to reach stability by using piecewise odd polynomial. In particular, if using piecewise odd polynomial and central numerical fluxes, the spatial convergence orders are just   $O(h^{k})$. In addition, as the numerical fluxes approach the central ones and the degree of polynomial decreases, the condition number of the coefficient matrix of the fully discrete scheme decreases gradually.

\section*{Appendix}
Equivalent form of Eq.~\eqref{first equation}.
Let $0<\alpha<1$ and $u(t,\x)\in C_{2,\delta}((0,T];H^{2}(\Omega))$. Then
\begin{equation}\label{appendix1}
	\partial_{t}^{\omega,\kappa(\x)}u(t,\x)=\Delta u(t,\x)-\kappa(\x)I_{t}^{\omega,\kappa(\x)}u(t,\x),
\end{equation}
is equivalent to
\begin{equation*}\label{appendix2}
	e^{-\kappa(\x)t}{}_0^C D_t^{\alpha}(e^{\kappa(\x)t}u(t,\x))=\Delta u(t,\x),
\end{equation*}
where $ C_{2,\delta}((0,T])= \{u(t)\big|\ |u^{(l)}|\leq C_u(1+t^{\delta-l}),\ l=0,1,2,\ 0<t\leq T\}$ and the regularity parameter $ \delta\in(0,1)\cup(1,2) $.
%and regularity parameter $ \delta\in(0,1)\cup(1,2) $.

\begin{proof}
	In the following, we define Laplace transform of $u(t,\x)$ about the time variable as
	\begin{equation*}\label{appendix2-1}
		\hat{u}(s,\x):=\mathscr{L}\{u(t,\x);s\}=\int_{0}^{\infty}e^{-st} u(t,\x)dt.
	\end{equation*}
	Let $\partial_{t}^{\omega,\kappa(x)}u(t,\x)=\frac{\partial }{\partial t}g(t,\x)$ in \eqref{TFD1}. According to the properties of Laplace transform, there exist
	\begin{equation*}\label{appendix3}
		\mathscr{L}\{\partial_{t}^{\omega,\kappa(x)}u(t,\x);s\}=s\hat{g}(s,\x)-g(0,\x)
	\end{equation*}
	and
	\begin{equation*}\label{appendix4}
		\hat{g}(s,\x)=\frac{1}{(s+\kappa(\x))^{1-\alpha}}(\hat{u}(s,\x)-\frac{1}{s}{u}(0,\x)).
	\end{equation*}
	
	Since $u(t,\x)\in C_{2,\delta}((0,T];H^{2}(\Omega))$, one can get
	\begin{equation*}\label{appendix5}
		\begin{split}
			\|g(0,\x)\|&= \lim_{t\rightarrow0^{+}}\frac{1}{\Gamma(1-\alpha)}\Big\|\int_{0}^{t}e^{-\kappa(\x)(t-\tau)}(t-\tau)^{-\alpha}(u(\tau,\x)-u(0,\x))d\tau\Big\|
			\\
			&\leq\lim_{t\rightarrow0^{+}}\frac{C_{max}}{\Gamma(1-\alpha)}\int_{0}^{t}(t-\tau)^{-\alpha}\|u(\tau,\x)-u(0,\x)\|d\tau
			\\
			&\leq\lim_{t\rightarrow0^{+}}\frac{2C_{max}C_{u}}{\Gamma(1-\alpha)}\int_{0}^{t}(t-\tau)^{-\alpha}(1+\tau^{\delta})d\tau
			\\
			&=0,
		\end{split}
	\end{equation*}
	where we denote $\|\cdot\| $ as the norms associated with $ L^2(\Omega) $ and $C_{max}$ is defined by \eqref{assumptions2}. One can further get $g(0,\x)=0 $.
Therefore
	\begin{equation}\label{appendix6}
		\mathscr{L}\{\partial_{t}^{\omega,\kappa(\x)}u(t,\x);s\}=\frac{1}{(s+\kappa(\x))^{1-\alpha}}(s\hat{u}(s,\x)-{u}(0,\x)).
	\end{equation}
Using inverse Laplace transform for \eqref{appendix6} results in
	\begin{equation}\label{appendix7}
		\partial_{t}^{\omega,\kappa(\x)}u(t,\x)=\int_{0}^{t}\frac{(t-\tau)^{-\alpha}}{\Gamma(1-\alpha)}e^{-\kappa(\x)(t-\tau)}\frac{\partial }{\partial  \tau}u(\tau,\x)d\tau.
	\end{equation}
	Substituting \eqref{appendix7} and \eqref{TFI} into \eqref{appendix1} leads to
	\begin{equation*}\label{appendix8}
		\begin{split}
			\Delta u(t,\x)&=
			\partial_{t}^{\omega,\kappa(\x)}u(t,\x)+\kappa(\x)I_{t}^{\omega,\kappa(\x)}u(t,\x)
			\\&=\frac{1}{\Gamma(1-\alpha)}\int_{0}^{t}e^{-\kappa(\x)(t-\tau)}(t-\tau)^{-\alpha}\Big[\frac{\partial }{\partial  \tau}u(\tau,\x)+\kappa(\x)u(\tau,\x)\Big]d\tau \\&=e^{-\kappa(\x)t}{}_0^C D_t^{\alpha}(e^{\kappa(\x)t}u(t,\x)),
		\end{split}
	\end{equation*}
	which concludes the proof.
\end{proof}

%=========================================================================================

%% For one-column wide figures use
%\begin{figure}
%% Use the relevant command to insert your figure file.
%% For example, with the graphicx package use
%%  \includegraphics{example.eps}
%\end{figure}
%%
%% For two-column wide figures use
%\begin{figure*}
%% Use the relevant command to insert your figure file.
%% For example, with the graphicx package use
%%  \includegraphics[width=0.75\textwidth]{example.eps}
%\end{figure*}
%

%\begin{acknowledgements}
%\end{acknowledgements}


\bibliographystyle{spmpsci}      % mathematics and physical sciences
\bibliography{name}   % name your BibTeX data base

%\begin{thebibliography}{}
%%
%% and use \bibitem to create references. Consult the Instructions
%% for authors for reference list style.
%\end{thebibliography}

\end{document}